\documentclass[12pt,a4paper]{article}
\usepackage{latexsym,amssymb,amsfonts,amsmath,amsthm,nccmath,float,enumitem,setspace,graphicx,bm}
\usepackage[hidelinks]{hyperref}
\usepackage[margin=2cm]{geometry}
\usepackage[usenames,dvipsnames]{xcolor}

\newtheorem{theorem}{Theorem}

\newtheorem{lemma}[theorem]{Lemma}

\theoremstyle{definition}

\def \deg {{\rm deg}}

\def \capacity {{\rm capacity}}

\def \l {\lambda}

\def \gcd {{\rm gcd}}
\def \leq {\leqslant}
\def \geq {\geqslant}

\def \mod#1{{\:({\rm mod}\ #1)}}

\let\oldproofname=\proofname
\renewcommand{\proofname}{\rm\bf{\oldproofname}}

\title{Decompositions of complete multigraphs into stars of varying sizes}

\author{
Rosalind A. Cameron$^a$ and Daniel Horsley$^b$}

\date{$^a$School of Mathematics and Statistics\\
University of Canterbury\\
Christchurch 8140, New Zealand\\
\texttt{rahoyte@outlook.com} \\[0.2cm]
$^b$School of Mathematics \\
Monash University \\
Vic 3800, Australia \\
\texttt{danhorsley@gmail.com}
}

\begin{document}
\maketitle\thispagestyle{empty}
\def\baselinestretch{1.2}\small\normalsize
\sloppy

\begin{abstract}
In 1979 Tarsi showed that an edge decomposition of a complete multigraph into stars of size $m$ exists whenever some obvious necessary conditions hold. In 1992 Lonc gave necessary and sufficient conditions for the existence of an edge decomposition of a (simple) complete graph into stars of sizes $m_1,\ldots,m_t$. We show that the general problem of when a complete multigraph admits a decomposition into stars of sizes $m_1,\ldots,m_t$ is $\mathsf{NP}$-complete, but that it becomes tractable if we place a strong enough upper bound on $\max(m_1,\ldots,m_t)$. We determine the upper bound at which this transition occurs. Along the way we also give a characterisation of when an arbitrary multigraph can be decomposed into stars of sizes $m_1,\ldots,m_t$ with specified centres, and a generalisation of Landau's theorem on tournaments.
\end{abstract}

\section{Introduction}

For a positive integer $m$, an $m$-star is a connected simple graph with $m$ edges, all of which are incident with a single vertex.
For $m \geq 2$, this vertex is unique and is called the \emph{centre} of the star. For some of the results in this paper it will be convenient to treat each star as having a unique centre, so we assume that $1$-stars have exactly one of their vertices designated as their centre. All the multigraphs considered in this paper will be loopless. Let $G$ be a multigraph. For distinct vertices $u$ and $v$ of $G$, we denote by $\mu_G(uv)$ the number of edges of $G$ between $u$ and $v$. A \emph{decomposition} $\mathcal{D}$ of $G$ is a collection of sub-multigraphs of $G$ such that $\sum_{H \in \mathcal{D}}\mu_H(uv)=\mu_G(uv)$ for all distinct vertices $u$ and $v$ of $G$. A \emph{packing} of $G$ is a decomposition of some sub-multigraph of $G$. We denote the $\l$-fold complete multigraph on $n$ vertices by $\l K_n$ and the $\l$-fold complete multigraph on vertex set $V$ by $\l K_V$. For a set $V$ we denote by $\binom{V}{2}$ the set $\{\{u,v\}:\mbox{$u,v \in V$ and $u \neq v$}\}$.

Tarsi \cite{Tarsi79} has shown that some obvious numerical necessary conditions for the existence of a decomposition of a complete multigraph into stars of a uniform specified size are also sufficient. The simple graph case of this result, along with the equivalent result on complete bipartite graphs, was independently proved by Yamamoto et al \cite{Yamamoto75}. Further, Lonc \cite{Lon} has given a simple numerical characterisation for the existence of a decomposition of a simple complete graph into stars of various specified sizes (Lonc in fact proved a more general result for uniform hypergraphs; his result for graphs was later rediscovered by Lin and Shyu \cite{LinShy96}). This paper deals with the common generalisation of these problems: when does a complete $\l$-fold multigraph admit an edge decomposition into stars of sizes $m_1,\ldots,m_t$? We show that this problem is $\mathsf{NP}$-complete if we allow $m_1,\ldots,m_t$ to take any values, but that the obvious necessary conditions for such a decomposition suffice if a suitable upper bound is placed on $\max(m_1,\ldots,m_t)$. It is worth noting that the analogous problems of when a complete $\l$-fold multigraph can be decomposed into matchings, paths or cycles of specified sizes have all been completely solved with numerical necessary and sufficient conditions (see \cite{Bar,Bry,BryHorMaeSmi}).

Problems concerning the decomposition of graphs into stars have been well studied. In addition to those already mentioned, {three} further results are {particularly} relevant to our purposes here. In \cite{Tarsi81} Tarsi showed a simple graph of order $n$ admits a decomposition into stars of sizes $m_1,\ldots,m_t$ provided its minimum degree is at least $\tfrac{n}{2}+\max(m_1,\ldots,m_t)-1$. In \cite{Hoffman04} Hoffman gave necessary and sufficient conditions for an arbitrary multigraph to have a decomposition into $m$-stars where the number of stars centred at each vertex is specified. In \cite{PriTar} it is shown that deciding whether an arbitrary $\lambda$-fold multigraph has a decomposition into stars of a uniform specified size is $\mathsf{NP}$-complete. See \cite{BryElZEynHof,Gardner,HofRob14} for other results on star decompositions.

Before stating our main result, we first formalise the primary question under investigation as a family of decision problems, one for each positive integer $\lambda$ and real number $\alpha$ such that $0 \leq \alpha \leq 1$.

\begin{description}[itemsep=0mm,parsep=0mm,topsep=1mm]
    \item[\textmd{\textsc{$(\lambda,\alpha)$-star decomp}}]

    \item[\textit{\textmd{Instance:}}]
Positive integers $n$ and $m_1,\ldots,m_t$ such that $\max(m_1,\ldots,m_t) \leq \alpha (n-1)$ and \mbox{$m_1+\cdots+m_t=\l\binom{n}{2}$}.

    \item[\textit{\textmd{Question:}}]
Does $\lambda K_n$ have a decomposition into stars of sizes $m_1,\ldots,m_t$?
\end{description}
Note that $m_1+\cdots+m_t=\l\binom{n}{2}$ and $\max(m_1,\ldots,m_t) \leq n-1$ are obvious necessary conditions for the existence of the required decomposition. Our main result is the following.

\begin{theorem}\label{Theorem:NPComplete}
Let $\lambda \geq 2$ be an integer. Then \textsc{$(\lambda,\alpha)$-star decomp} is $\mathsf{NP}$-complete if $\alpha>\alpha'$, where
\[\alpha'=
\left\{
  \begin{array}{ll}
    \tfrac{\lambda}{\lambda+1}, & \hbox{if $\lambda$ is odd;} \\
    1 - \frac{2}{\lambda}(3-2\sqrt{2}), & \hbox{if $\lambda$ is even.}
  \end{array}
\right.\]
Furthermore, if $\alpha \leq \alpha'$ then every instance of \textsc{$(\lambda,\alpha)$-star decomp} is feasible and the required decompositions can be constructed in polynomial time.
\end{theorem}

Our major tool in proving Theorem~\ref{Theorem:NPComplete} is a result which concerns the problem of packing an arbitrary multigraph $G$ with stars, where the sizes of the stars to be centred at each vertex are specified.
We give a characterisation of when such a packing exists and show there is a polynomial time (in the problem's input size) algorithm for deciding this. The input $I$ to this problem involves specifying the multiplicity function $\mu_G$ of $G$ represented by a sequence of $\binom{n}{2}$ nonnegative integers, and the sequence of star sizes $m_{v,1},\ldots,m_{v,|M_v|}$ for each $v\in V(G)$. Thus the size of the input $|I|$ is bounded below as follows.
\begin{equation}\label{Equation:InputSize}
|I| \geq \medop\sum_{\{ u,v\}\in \binom{V}{2}}\log (\mu_G(uv)+1)+\medop\sum_{v\in V}\medop\sum_{i=1}^{|M_v|}\log (m_{v,i}+1)
\end{equation}

Of course, if the sum of the star sizes is equal to the number of edges in $G$ (counting multiplicities) then a decomposition will result, and so our result also characterises decompositions. It generalises the result of Hoffman mentioned above.

For a multiset $M$ of positive integers we define $\sigma(M)$ to be the sum of all the elements of $M$ and for each $i \in \{0,\ldots,|M|\}$ we define $\sigma_i(M)$ to be the sum of the largest $i$ elements in $M$.

\begin{theorem}\label{Theorem:RealTruthLambda}
Let $G$ be a loopless multigraph with vertex set $V$ and, for each $v \in V$, let $M_v$ be a (possibly empty) multiset of positive integers. There is a packing of $G$ with stars such that $M_v$ is the multiset of sizes of stars centred at $v$ for each $v \in V$ if and only if
\[\sum_{v \in V} \sigma_{f(v)}(M_v) \leq \sum_{\{u,v\} \in \binom{V}{2}}\min(f(u)+f(v),\mu_G(uv))\]
for each function $f: V \rightarrow \mathbb{Z}$ such that $f(v) \in \{0,\ldots,|M_v|\}$ for each $v \in V$. Furthermore, there is a polynomial time (in the problem's input size) algorithm for deciding whether such a packing exists, and constructing one if so.
\end{theorem}

We introduce some notation relating to Theorem~\ref{Theorem:RealTruthLambda}. Let $\mathcal{G}$ be a loopless multigraph $G$ with vertex set $V$ equipped with multisets $\{M_v:v \in V\}$ of integers. Call a packing of $G$ with stars a \emph{star $\mathcal{G}$-packing} if $M_v$ is the multiset of sizes of stars centred at $v$ for each $v \in V$. We say a function $f: V \rightarrow \mathbb{Z}$ such that $f(v) \in \{0,\ldots,|M_v|\}$ for each $v \in V$ is a \emph{restriction function} for $\mathcal{G}$. For a restriction function $f$ for $\mathcal{G}$, we define $\Delta^-_f(\mathcal{G})=\sum_{v \in V} \sigma_{f(v)}(M_v)$, $\Delta^+_f(\mathcal{G})=\sum_{\{u,v\} \in \binom{V}{2}}\min(f(u)+f(v),\mu_G(uv))$, and $\Delta_f(\mathcal{G})=\Delta^+_f(\mathcal{G})-\Delta^-_f(\mathcal{G})$. Finally, we define $\Delta(\mathcal{G})$ to be the minimum value of $\Delta_f(\mathcal{G})$ over all restriction functions $f$ for $\mathcal{G}$. We say a restriction function $f$ is \emph{minimal} if $\Delta_f(\mathcal{G})=\Delta(\mathcal{G})$. Considering the restriction function that is uniformly $0$, we always have $\Delta(\mathcal{G}) \leq 0$. Theorem~\ref{Theorem:RealTruthLambda} effectively states that a star $\mathcal{G}$-packing exists if and only if $\Delta(\mathcal{G}) = 0$.

Intuitively, $\Delta^-_f(\mathcal{G})$ is the number of edges in the sub-multigraph of $G$ induced by the largest $f(v)$ stars centred at each vertex $v \in V$ and $\Delta^+_f(\mathcal{G})$ is the number of edges in the multigraph obtained from $G$ by limiting the number of edges between $u$ and $v$ to $f(u)+f(v)$ for each $\{u,v\} \in \binom{V}{2}$. The former multigraph must be a sub-multigraph of the latter, and so the necessity of the condition $\Delta^-_f(\mathcal{G}) \leq \Delta^+_f(\mathcal{G})$ is obvious.

Theorem~\ref{Theorem:RealTruthLambda} has a consequence concerning tournaments. For a vertex set $V$, a \emph{$\lambda$-fold tournament on $V$} is a graph produced by orienting the edges of $\l K_V$. For a vertex $v$ of an oriented multigraph $G$, let $\deg^+_G(v)$ denote the number of edges that are incident with $v$ and oriented out from it, and let $N^+_G(v)$ denote the set of all $w \in V(G)$ for which there is at least one edge of $G$ oriented from $v$ to $w$.

In \cite{Landau} Landau characterised when there exists a $1$-fold tournament with a specified out-degree at each vertex. This result generalises easily to $\lambda$-fold tournaments (see \cite[Theorem 2.2.4]{Brualdi06} or \cite{BruFri15}, for example). Using Theorem~\ref{Theorem:RealTruthLambda}, we can prove a further generalisation to Landau's theorem in which we also specify a lower bound on the size of the out-neighbourhood of each vertex.

\begin{theorem}\label{Theorem:OutNeighbourhoodTournament}
Let $V$ be a set of $n$ vertices and let $a: V \rightarrow \mathbb{Z}$ and $b: V \rightarrow \mathbb{Z}$ be functions such that $a(v) \geq b(v) \geq 0$ for each $v \in V$ and $\sum_{v \in V}a(v)=\l\binom{n}{2}$. There exists a $\l$-fold tournament $T$ on $V$ such that $\deg^+_T(v) = a(v)$ and $|N^+_T(v)| \geq b(v)$ for each $v \in V$ if and only if, for each $k \in \{0,\ldots,n-1\}$,
\[\psi_k+\sum_{v \in V}b_k(v) \leq \tfrac{1}{2}\lambda k(2n-k-1),\]
where $b_k(v)=\max(0,b(v)-n+k+1)$ for each $v \in V$ and $\psi_k$ is the sum of the greatest $k$ elements of the multiset $\{a(v)-b_k(v): v \in V\}$.
\end{theorem}

The generalisation of Landau's result to $\l$-fold tournaments can be recovered from Theorem~\ref{Theorem:OutNeighbourhoodTournament} by setting $b(v)=0$ for each $v \in V$, noting that in this case $\sum_{v \in V}b_k(v)=0$ and $\psi_k$ is equal to the sum of the greatest $k$ elements of $\{a(v):v\in V\}$.

\section{Proof of Theorems~\ref{Theorem:RealTruthLambda} and \ref{Theorem:OutNeighbourhoodTournament}}

As discussed, Hoffman \cite{Hoffman04} obtained a characterisation for the existence of a decomposition of an arbitrary multigraph into uniform size stars, where the number of stars centred at each vertex is specified. His proof relies on constructing an equivalent network flow problem. We now extend this idea to prove Theorem~\ref{Theorem:RealTruthLambda}.

In the rest of the paper, we often use the exponential notation $\{x_1^{[e_1]},\ldots,x_t^{[e_t]}\}$ to describe multisets, where, for each $i \in \{1,\ldots,t\}$, $x_i^{[e_i]}$ stands for $e_i$ occurrences of $x_i$.

\begin{proof}[\textbf{\textup{Proof of Theorem~\ref{Theorem:RealTruthLambda}.}}]
Let $\mathcal{G}$ be the multigraph $G$ equipped with the multisets $\{M_v:v\in V\}$. Let $|I(\mathcal{G})|$ be the input size of $\mathcal{G}$ as discussed before the statement of Theorem~\ref{Theorem:RealTruthLambda}.

Let $S=\{(u,i):u \in V,i \in \{1,\ldots,|M_u|\}\}$ and $T=\{\{u,v\} \in \binom{V}{2}: \mu_G(uv) > 0\}$. For each $u \in V$, let $M_u=\{m_{u,1},\ldots,m_{u,|M_u|}\}$, and let $z=\sum_{(u,i) \in S}m_{u,i}$. We will establish an equivalence between packings of $G$ satisfying the conditions of Theorem~\ref{Theorem:RealTruthLambda} and integer flows of magnitude $z$ through the flow network $N$ composed of
\begin{itemize}
    \item
a source $a$ and a sink $b$;
    \item
an internal vertex $s_{(u,i)}$ for all $(u,i) \in S$;
    \item
an internal vertex $t_{\{u,v\}}$ for all $\{u,v\} \in T$;
    \item
an arc $as_{(u,i)}$ with capacity $m_{u,i}$ for all $(u,i) \in S$;
    \item
an arc $s_{(u,i)}t_{\{u,v\}}$ with capacity 1 for all $(u,i) \in S$ and $v \in V \setminus \{u\}$ such that $\{u,v\} \in T$;
    \item
an arc $t_{\{u,v\}}b$ with capacity $\mu_G(uv)$ for all $\{u,v\} \in T$.
\end{itemize}
With any integer flow of magnitude $z$ through $N$ we can associate a multiset of stars $\mathcal{P}=\{H_{u,i}:(u,i) \in S\}$ where, for $(u,i) \in S$, $H_{u,i}$ is a star centred at $u$ with $E(H_{u,i})=\{uv:\hbox{arc $s_{(u,i)}t_{\{u,v\}}$ has flow 1}\}$. Note that $\mathcal{P}$ is a packing of $G$ because, for each $\{u,v\} \in T$, the number of stars in $\mathcal{P}$ using an edge between $u$ and $v$ is exactly the flow through the arc $t_{\{u,v\}}b$ in $N$ and hence is at most $\mu_G(uv)$. Also, for each $(u,i) \in S$, $|E(H_{u,i})|=m_{u,i}$ because any flow of magnitude $z$ through $N$ must have flow exactly $m_{u,i}$ through arc $as_{(u,i)}$. Thus, $\mathcal{P}$ is a packing satisfying the conditions of Theorem~\ref{Theorem:RealTruthLambda}. Conversely any packing satisfying the conditions of Theorem~\ref{Theorem:RealTruthLambda} can be associated with an integer flow of magnitude $z$ through $N$.

Given this equivalence, it suffices to show that there is a flow of magnitude $z$ through $N$ if and only if the hypotheses of the theorem hold. Hence, by the max-flow min-cut theorem (and the integer flow theorem), it suffices to show that a minimum capacity cut of $N$ has capacity at least $z$ if and only if the hypotheses of the theorem hold. Note that establishing this will immediately provide the polynomial time algorithm whose existence the theorem asserts. This is because the number of vertices in $N$ is at most $2+|T|+\sum_{v \in V}|M_v|$ which is polynomial in $|I(\mathcal{G})|$ by \eqref{Equation:InputSize}, and it is well known that there is a polynomial time (in the number of vertices in the network) algorithm for finding an integral maximum flow in a network.

With each cut $(A^*,B^*)$ of $N$ where $a \in A^*$ and $b \in B^*$, we associate the restriction function $f^*:V \rightarrow \mathbb{Z}$ for $\mathcal{G}$ given by $f^*(u)=|\{i \in \{1,\ldots,|M_u|\}:s_{(u,i)} \in A\}|$. Now let $f:V \rightarrow \mathbb{Z}$ be a fixed restriction function for $\mathcal{G}$. Note that there is at least one cut of $N$ whose associated restriction function is $f$ and, of all such cuts, let $(A,B)$ be one of minimum capacity. The capacity of $(A,B)$ is
\begin{align*}
{}& \sum_{s_{(u,i)} \in A,\ t_{\{v,w\}} \in B}\capacity(s_{(u,i)}t_{\{v,w\}})+\sum_{t_{\{v,w\}} \in A}\capacity(t_{\{v,w\}}b)+\sum_{s_{(u,i)} \in B}\capacity(as_{(u,i)})\\[1mm]
={}& \sum_{t_{\{v,w\}} \in B}\big(f(v)+f(w)\big)+\sum_{t_{\{v,w\}} \in A}\mu_G(vw)+\sum_{s_{(u,i)} \in B}m_{u,i}\\[1mm]
={}& \left(\sum_{t_{\{v,w\}} \in B}\big(f(v)+f(w)\big)+\sum_{t_{\{v,w\}} \in A}\mu_G(vw)\right)+z-\sum_{s_{(v,i)} \in A}m_{v,i}\\[1mm]
={}& \sum_{\{v,w\} \in T}\min\big(f(v)+f(w),\mu_G(vw)\big)+z-\sum_{u \in V}\sigma_{f(u)}(M_u).
\end{align*}
The last equality follows because the minimality of $(A,B)$ implies that  $\{m_{u,i}:s_{(u,i)} \in A\}$ is the multiset of the $f(u)$ largest elements in $M_u$ for each $u \in V$ and that, for each $\{v,w\} \in T$, $t_{\{v,w\}} \in B$ if $f(v)+f(w) < \mu_G(vw)$ and $t_{\{v,w\}} \in A$ if $\mu_G(vw) < f(v)+f(w)$. So $(A,B)$ has capacity at least $z$ if and only if
\[\sum_{u \in V} \sigma_{f(u)}(M_u) \leq \sum_{\{v,w\} \in T}\min(f(v)+f(w),\mu_G(vw)).\]

So if this inequality holds for all restriction functions, then each cut of $N$ has capacity at least $z$. Conversely, if the inequality fails for some restriction function, then there is a cut of $N$ with capacity less than $z$.
\end{proof}

We next prove Lemma~\ref{Lemma:RestrictionRestriction}, which is a simple result on minimal restriction functions.

\begin{lemma}\label{Lemma:RestrictionRestriction}
Let $\mathcal{G}$ be a multigraph $G$ equipped with multisets $\{M_v:v\in V(G)\}$ of positive integers. Suppose there is a minimal restriction function $f_j$ for $\mathcal{G}$ such that $f_j(w)=j$ for some $w \in V(G)$ and integer $j$ such that $1\leq j\leq |M_w|-1$. Let $m_j$ and $m_{j+1}$ be the $j$th and $(j+1)$st largest elements of $M_w$ and, for each $i \in \{j-1,j+1\}$, let $f_i$ be the restriction function for $\mathcal{G}$ such that $f_i(v)=f_j(v)$ for each $v \in V(G) \setminus \{w\}$ and $f_i(w)=i$.
\begin{itemize}
    \item[\textup{(a)}]
If $m_{j+1}=m_j$, then $f_{j-1}$ is also minimal.
    \item[\textup{(b)}]
If $m_{j+1}=m_j-1$, then one of $f_{j-1}$ or $f_{j+1}$ is also minimal.
\end{itemize}
\end{lemma}

\begin{proof}
For each $i \in \{j,j+1\}$, observe that $\Delta^-_{f_i}(\mathcal{G}) = \Delta^-_{f_{i-1}}(\mathcal{G})+m_i$ and let $k_i$ be the integer such that $\Delta^+_{f_i}(\mathcal{G})=\Delta^+_{f_{i-1}}(\mathcal{G})+k_i$. Thus $\Delta_{f_i}(\mathcal{G}) = \Delta_{f_{i-1}}(\mathcal{G})+k_i-m_i$ for each $i \in \{j,j+1\}$. Because $f_j$ is minimal, $\Delta_{f_j}(\mathcal{G}) \leq \Delta_{f_i}(\mathcal{G})$ for each $i \in \{j-1,j+1\}$, and so $m_{j+1} \leq k_{j+1}$ and $k_{j} \leq m_{j}$. Now, it can be seen from the definition of $\Delta^+_f(\mathcal{G})$ that $k_{j+1} \leq k_j$, and so $m_{j+1} \leq k_{j+1} \leq k_{j} \leq m_{j}$. Thus, if $m_{j+1}=m_j$, then $k_j=m_j$ and (a) follows. Similarly (b) follows because if $m_{j+1}=m_j-1$, then $k_j=m_j$ or $k_{j+1}=m_{j+1}$.
\end{proof}

It is not immediately apparent that Hoffman's result \cite[Theorem 1]{Hoffman04} follows from Theorem~\ref{Theorem:RealTruthLambda}. However, by Lemma~\ref{Lemma:RestrictionRestriction}(a), in the case where all the prescribed star sizes are equal it suffices to consider only restriction functions such that $f(v) \in \{0,|M_v|\}$ for each vertex $v$, and so Theorem~\ref{Theorem:RealTruthLambda} reduces to Hoffman's theorem. To see that considering only such restriction functions in Theorem~\ref{Theorem:RealTruthLambda}  does not suffice in general, consider taking $\mathcal{G}$ to be $2K_{10}$ where two vertices are equipped with multisets $\{9,5\}$, four vertices with $\{9,1\}$, and the remaining four with $\{5\}$. The restriction function which takes values $2$, $1$ and $0$ at the first, second and third type of vertices respectively shows that a star $\mathcal{G}$-decomposition does not exist, but the same is not true for any restriction function $f$ such that $f(v) \in \{0,|M_v|\}$ for each vertex $v$.

We conclude this section by proving Theorem~\ref{Theorem:OutNeighbourhoodTournament} which we achieve with the help of the following lemma.

\begin{lemma}\label{Lemma:OutNeighbourhoodTournament}
Let $V$ be a set of $n$ vertices and let $a: V \rightarrow \mathbb{Z}$ and $b: V \rightarrow \mathbb{Z}$ be functions such that $a(v) \geq b(v) \geq 0$ for each $v \in V$. There exists a $\l$-fold tournament $T$ on $V$ such that $\deg^+_T(v) = a(v)$ and $|N^+_T(v)| \geq b(v)$ for each $v \in V$ if and only if, for all disjoint subsets $A$ and $B$ of $V$,
\begin{equation}
\sum_{v\in A}a(v)+\sum_{v\in B}b(v) \leq \tfrac{1}{2}\lambda |A|(2n-|A|-1)+|B|(n-|A|-1) \label{Equation:OutNeighbourhoodLemmaCond}
\end{equation}
with equality in the case $(A,B)=(V,\emptyset)$.
\end{lemma}

\begin{proof}
We may assume that $\sum_{v\in V}a(v) = \tfrac{1}{2}\lambda n(n-1)$ for otherwise the condition of the lemma does not hold when $(A,B)=(V,\emptyset)$ and clearly there is no $\l$-fold tournament $T$ on $V$ such that $\deg^+_T(v) = a(v)$ for each $v \in V$.

For each $v \in V$, let $M_v={\{b(v),1^{[a(v)-b(v)]}\}}$ if $b(v) \geq 2$ and $M_v={\{1^{[a(v)]}\}}$ if $b(v) \in \{0,1\}$. By considering the edges of stars to be oriented outward from their centres, it can be seen that a $\l$-fold tournament $T$ on $V$ such that $\deg^+_T(v) = a(v)$ and $|N^+_T(v)| \geq b(v)$ for each $v \in V$ exists if and only if a star $\mathcal{K}$-decomposition exists, where $\mathcal{K}$ is $\l K_V$ equipped with the multisets $\{M_v:v \in V\}$.

Observe that if $A$ and $B$ are disjoint subsets of $V$ and $f$ is a restriction function for $\mathcal{K}$ such that $f(v)=|M_v|$ for each $v \in A$, $f(v)=1$ for each $v \in B$ and $f(v)=0$ for each $v \in V \setminus (A\cup B)$, then
\begin{align}
\Delta^+_f(\mathcal{K}) &\leq
\lambda\mbinom{|A|}{2}+\lambda |A| (n-|A|)+2\mbinom{|B|}{2}+|B|(n-|A|-|B|) \label{Equation:OutNeighbourhoodLemma1}\\
\Delta^-_f(\mathcal{K}) &\geq\medop\sum_{v\in A}a(v)+\medop\sum_{v\in B}b(v) \label{Equation:OutNeighbourhoodLemma2}.
\end{align}
Furthermore, we have equality in \eqref{Equation:OutNeighbourhoodLemma1} if $|M_v| \geq \lambda$ for each $v \in A$ and we have equality in \eqref{Equation:OutNeighbourhoodLemma2} if $b(v) \neq 0$ for each $v \in B$. Note also that the right hand side of \eqref{Equation:OutNeighbourhoodLemma2} is equal to the left hand side of \eqref{Equation:OutNeighbourhoodLemmaCond} and, by routine calculation, the right hand side of \eqref{Equation:OutNeighbourhoodLemma1} is equal to the right hand side of \eqref{Equation:OutNeighbourhoodLemmaCond}.

To prove the `only if' direction, suppose there are disjoint subsets $A$ and $B$ of $V$ for which \eqref{Equation:OutNeighbourhoodLemmaCond} fails. Let $f$ be a restriction function for $\mathcal{K}$ such that $f(v)=|M_v|$ for each $v \in A$, $f(v)=1$ for each $v \in B$ and $f(v)=0$ for each $v \in V \setminus (A\cup B)$. Then by \eqref{Equation:OutNeighbourhoodLemma1}, \eqref{Equation:OutNeighbourhoodLemma2} and the failure of \eqref{Equation:OutNeighbourhoodLemmaCond}, we have $\Delta^-_f(\mathcal{K}) > \Delta^+_f(\mathcal{K})$ and hence that no star $\mathcal{K}$-decomposition exists by Theorem~\ref{Theorem:RealTruthLambda}.

To prove the `if' direction, suppose that \eqref{Equation:OutNeighbourhoodLemmaCond} holds for all disjoint subsets $A$ and $B$ of $V$. Of all the minimal restriction functions for $\mathcal{K}$, let $f$ be one that maximises the sequence $(|f^{-1}(0)|,|f^{-1}(1)|,\ldots)$ lexicographically. We claim that, for each $v \in V$, $f(v) \in \{0,1,|M_v|\}$, $f(v) \neq |M_v|$ if $2 \leq |M_v| \leq \l-1$, and $f(v) \neq 1$ if $b(v)=0$. Proving this claim will suffice to prove the lemma because, if we set $A=\{v:f(v) \geq \l\}$ and $B=\{v:\mbox{$f(v)=1$}\}$, then \eqref{Equation:OutNeighbourhoodLemma1} and \eqref{Equation:OutNeighbourhoodLemma2} hold with equality, and hence $\Delta(\mathcal{K})=\Delta^+_f(\mathcal{K}) - \Delta^-_f(\mathcal{K}) \geq 0$ because \eqref{Equation:OutNeighbourhoodLemmaCond} holds. By Lemma~\ref{Lemma:RestrictionRestriction}(a) it can be seen that, for each $v \in V$, $f(v) \in \{0,1,|M_v|\}$,  and $f(v) \neq 1$ if $a(v) \geq 2$ and $b(v) =0$. We will complete the proof by showing that in addition, for each $v \in V$, $f(v) \neq |M_v|$ if $2 \leq |M_v| \leq \l-1$ or if $M_v=\{1\}$.

Suppose otherwise and, of all the elements of
\[\{v \in V: \mbox{$f(v)=|M_v|$ and either $2 \leq f(v) \leq \l-1$ or $M_v=\{1\}$}\},\]
let $w$ be one with a minimum value of $|M_w|$.
If $f(v) \geq \l-|M_w|+1$ for each $v\in V\setminus \{w\}$, then $f(v) \geq \max (|M_w|,\lambda-|M_w|+1)$ for all $v\in V \setminus \{w\}$ by the definition of $w$ and we can conclude successively that $\Delta^+_f(\mathcal{K})=\lambda \binom{n}{2}$, that $\Delta_f(\mathcal{K})=0$  (because $\Delta^-_f(\mathcal{K}) \leq \lambda \binom{n}{2}$ and we always have $\Delta(\mathcal{K}) \leq 0$), and the contradiction that $f$ is uniformly 0 (recall that $f(w)=|M_w|>0$).
So there is some $u\in V\setminus \{w\}$ such that $f(u) \leq \l-|M_w|$. Let $f_1$ be the restriction function such that $f_1(v) = f(v)$ for each $v \in V \setminus \{w\}$ and $f_1(w)=|M_w|-1$. Then $\Delta^+_{f_1}(\mathcal{K}) \leq \Delta^+_{f}(\mathcal{K})-1$ because $u$ exists and $\Delta^-_{f_1}(\mathcal{K})=\Delta^-_{f}(\mathcal{K})-1$ by our definitions of $\mathcal{K}$ and $w$. Thus $\Delta_{f_1}(\mathcal{K}) \leq \Delta_{f}(\mathcal{K})$, $|f_1^{-1}(i)|=|f^{-1}(i)|$ for each $i \in \{0,\ldots,|M_w|-2\}$ and $|f_1^{-1}(|M_w|-1)|>|f^{-1}(|M_w|-1)|$, which contradicts our definition of $f$.
\end{proof}

\begin{proof}[\textup{\textbf{Proof of Theorem~\ref{Theorem:OutNeighbourhoodTournament}}}]
In Lemma~\ref{Lemma:OutNeighbourhoodTournament}, for a fixed choice of $A$, it can be seen that the inequality is tightest when $B=\{v \in V \setminus A: b(v) \geq n-k\}$ where $k=|A|$. Thus, for a fixed choice of $A$ and $k=|A|$, a tournament satisfying the conditions of Lemma~\ref{Lemma:OutNeighbourhoodTournament} exists if and only if
\[
\sum_{v\in A}a(v) + \sum_{v \in V \setminus A}b_k(v) \leq \tfrac{1}{2}\lambda k(2n-k-1),
\]
or equivalently,
\begin{equation}\label{Equation:OutNeighbourhoodTournament2}
\sum_{v\in A}(a(v)-b_k(v))+\sum_{v \in V}b_k(v) \leq \tfrac{1}{2}\lambda k(2n-k-1).
\end{equation}
Now, for a fixed choice $k$ of $|A|$, the right hand side and the second sum on the left hand side in \eqref{Equation:OutNeighbourhoodTournament2} are constant and the maximum value of the first sum on the left hand side is exactly $\psi_k$. The result follows.
\end{proof}

\section{Theorem~\ref{Theorem:NPComplete} proof strategy}\label{Section:Strategy}

We begin this section with two very simple results that will be useful in the proof of Theorem~\ref{Theorem:NPComplete}.

\begin{lemma}\label{Lemma:bloodyObvious}
Let $m_1,\ldots,m_t,x,y$ be positive integers. If there is a packing of a multigraph $G$ with stars of sizes $m_1,\ldots,m_t,x+y$ then there is a packing of $G$ with stars of sizes $m_1,\ldots,m_t,x,y$.
\end{lemma}

\begin{proof}
Begin with the packing of $G$ with stars of sizes $m_1,\ldots,m_t,x+y$ and replace a star $H$ of size $x+y$ with two stars $H_1$ and $H_2$ such that $|E(H_1)|=x$, $|E(H_2)|=y$ and $\{H_1,H_2\}$ is a decomposition of $H$.
\end{proof}

Call a multigraph $G$ a \emph{multistar} if $|V(G)| \geq 2$, $G$ is connected and $G$ has some vertex $c$ with which every edge is incident. For $|V(G)| \geq 3$, this vertex is unique and is called the \emph{centre} of the multistar. When $|V(G)|=2$ we assume that one of the vertices is designated as the centre.

\begin{lemma}\label{Lemma:multistarDecomp}
Let $G$ be a multistar with centre $c$ and let $m_1,\ldots,m_t$ be positive integers such that $m_1 \geq \cdots \geq m_t$. There is a packing of $G$ with stars of sizes $m_1,\ldots,m_t$ if and only if, for each $s \in \{1,\ldots,t\}$,
\[\sum_{i=1}^{s}m_i \leq \sum_{v \in V(G) \setminus \{c\}}\min(s,\mu_G(cv)).\]
\end{lemma}

\begin{proof}
In any packing of $G$ with stars of sizes $m_1,\ldots,m_t$, each star of size greater than one must be centred at $c$ and we may assume without loss of generality that each star of size 1 is centred at $c$. The lemma is now a specialisation of Theorem~\ref{Theorem:RealTruthLambda}.
\end{proof}

In Sections~\ref{Section:LambdaOdd} and \ref{Section:LambdaEven} we will prove Theorem~\ref{Theorem:NPComplete} in the cases where $\lambda$ is odd and even, respectively. Here we discuss our overall proof strategy. Theorem~\ref{Theorem:RealTruthLambda} is our main tool in proving Theorem~\ref{Theorem:NPComplete}.  For each parity of $\lambda$, we first show that \textsc{$(\lambda,\alpha)$-star decomp} is $\mathsf{NP}$-complete when $\alpha>\alpha'$, and then show that, when $\alpha \leq\alpha'$, {every instance of \textsc{$(\lambda,\alpha)$-star decomp} is feasible.}{} {The plethora of possible restriction functions can be an obstacle to exploiting Theorem~\ref{Theorem:RealTruthLambda}. To deal with this we show that, when the multisets assigned to the vertices of $\l K_n$ are well-behaved in certain ways, there must be a minimal restriction function of a particular form (see Lemmas~\ref{Lemma:MinimalFProperties} and \ref{Lemma:MinimalFProperties1LambdaEven}).}{}

To establish the $\mathsf{NP}$-completeness of \textsc{$(\lambda,\alpha)$-star decomp} when $\alpha>\alpha'$ we will reduce to it from the decision problem \textsc{3-partition}.
\begin{description}[itemsep=0mm,parsep=0mm,topsep=1mm]
    \item[\textmd{\textsc{3-partition}}]

    \item[\textit{\textmd{Instance:}}]
A multiset $\{a_1,\ldots,a_{3q}\}$ of positive integers such that $a=\frac{1}{q}(a_1+\cdots+a_{3q})$ is an integer and $\frac{a}{4} < a_i < \frac{a}{2}$ for each $i \in \{1,\ldots,3q\}$.

    \item[\textit{\textmd{Question:}}]
Is there a partition of $\{a_1,\ldots,a_{3q}\}$ into $q$ classes such that the elements of each class sum to $a$?
\end{description}
It is known that \textsc{3-partition} is $\mathsf{NP}$-complete in the strong sense; that is, it remains $\mathsf{NP}$-complete even when $qa$ is bounded by a polynomial in the length of its input (see \cite[Theorem~4.2]{GarJoh}). This fact means that it suffices for us to reduce from it to an instance of \textsc{$(\lambda,\alpha)$-star decomp} whose input size is polynomial in $qa$.

Our strategy for showing that {every instance of \textsc{$(\lambda,\alpha)$-star decomp} is feasible} when $\alpha \leq\alpha'$ is as follows. We first set an upper bound $m$ on the star size, where $m$ is equal to (or slightly larger than) $\lfloor \alpha'(n-1) \rfloor$. We note that by Lemma~\ref{Lemma:bloodyObvious}, we may assume that any two distinct specified star sizes sum to more than $m$. Next, we assign centre vertices to the specified star sizes, resulting in a multiset $M_v$ of star sizes to be centred at each vertex $v$ of our complete multigraph $\l K_n$. We then ``compress'' each multiset $M_v$ into a new multiset $M^*_v$ such that  $\sigma(M^*_v)=\sigma(M_v)$ and $\sigma_i(M^*_v) \geq \sigma_i(M_v)$ for each $i \in \{1,\ldots,|M^*_v|\}$. Let $\mathcal{K}$ be $\l K_n$ equipped with the multiset $M^*_v$ at each vertex $v$. It follows by Lemma~\ref{Lemma:multistarDecomp} that, in a star $\mathcal{K}$-decomposition, the multistar induced by the stars centred at any vertex $v$ has a decomposition into stars of sizes given by the elements of $M_v$. Thus it suffices to show there exists a star $\mathcal{K}$-decomposition. Finally, we apply Theorem~\ref{Theorem:RealTruthLambda}. The compression ensures that the assigned multisets are well-behaved and hence, as discussed, the existence of a minimal restriction function of a particular form. Using this, we are able to conclude that $\Delta(\mathcal{K}) \geq 0$ and hence that the desired decomposition exists. Ensuring that distinct star sizes sum to more than $m$, compressing multisets, and applying Theorem~\ref{Theorem:RealTruthLambda} to construct a suitable decomposition can all be completed in polynomial time in $n$. We will show that the same is true for the procedures by which we assign the star sizes to the vertices.

We introduce some notation related to multisets that will be used throughout the rest of the paper. For a multiset $M$ of positive integers and a positive integer $x$ we define $\nu_x(M)$ to be the number of elements of $M$ equal to $x$, and for a set $S$ of positive integers we define $\nu_S(M)=\sum_{x \in S}\nu_x(M)$. For multisets $M$ and $N$ of positive integers, we say $M \subseteq N$ if $\nu_{x}(M) \leq \nu_{x}(N)$  for all positive integers $x$, and we define $M \uplus N$ and $M \setminus N$ so that, for all positive integers $x$, $\nu_{x}(M \uplus N)=\nu_{x}(M)+\nu_{x}(M)$ and $\nu_{x}(M \setminus N)=\max(0,\nu_{x}(M)-\nu_{x}(N))$.

We now give two technical lemmas which will be useful in Sections~\ref{Section:LambdaOdd} and \ref{Section:LambdaEven}. Lemma~\ref{Lemma:partitionV} is used in establishing the \textsf{NP}-completeness of \textsc{$(\lambda,\alpha)$-star decomp} when $\alpha>\alpha'$ whereas Lemma~\ref{Lemma:dualThing} is used in proving that every instance of  \textsc{$(\lambda,\alpha)$-star decomp} is feasible when $\alpha \leq \alpha'$.

\begin{lemma}\label{Lemma:partitionV}
Let $\mathcal{K}$ be a complete multigraph $\lambda K_V$ equipped with multisets $\{M_v: v \in V\}$ of positive integers and let $n=|V|$. Let $\{V',V''\}$ be a partition of $V$ such that $|V'|=q$, let $M'=\biguplus_{v \in V'}M_v$ and let $M''=\biguplus_{v \in V''}M_v$. If there is a star $\mathcal{K}$-decomposition, then
\begin{itemize}
    \item[\textup{(a)}]
$\sum_{x \in M''}(x-q) \leq \lambda\binom{n-q}{2}$;
    \item[\textup{(b)}]
if $\sum_{x \in M''}(x-q) = \lambda\binom{n-q}{2}$, then $\sigma(M_u) \leq \lambda(n-1)-|M''|$ for each $u \in V'$.
\end{itemize}
\end{lemma}

\begin{proof}
Let $S$ be a subset of $V$. Applying Theorem~\ref{Theorem:RealTruthLambda} with $f(x)=0$ for each $x \in S$ and $f(v)=|M_{v}|$ for each $v \in V \setminus S$, we have that
\begin{equation}\label{Equation:partitionV}
\sum_{v \in V \setminus S}\sigma(M_{v}) \leq \Delta^+_f(\mathcal{K}) \leq \lambda\mbinom{n-|S|}{2}+|S|\sum_{v \in V \setminus S}\min(\lambda,|M_{v}|).
\end{equation}
Now (a) follows by setting $S=V'$ in \eqref{Equation:partitionV}, using $\min(\lambda,|M_{v}|) \leq |M_v|$ for each $v \in V''$, and subtracting $q|M''|$ from each side of the inequality.

For each $u \in V'$, by \eqref{Equation:partitionV} with $S=V'\setminus \{u\}$, using $\min(\lambda,|M_{v}|) \leq |M_v|$ for each $v \in V''$ and $\min(\lambda,|M_{u}|) \leq \lambda$,
\begin{equation}\label{Equation:partitionV2}
\sum_{v \in V'' \cup \{u\}}\sigma(M_v) \leq  \lambda\mbinom{n-q+1}{2}+(q-1)\lambda+(q-1)|M''|.
\end{equation}
If $\sum_{x \in M''}(x-q) = \lambda\binom{n-q}{2}$, then $\sigma(M'') =\lambda\binom{n-q}{2}+ q|M''|$. Now (b) follows by subtracting this equation from \eqref{Equation:partitionV2}.
\end{proof}

\begin{lemma}\label{Lemma:dualThing}
Let $\mathcal{K}$ be a complete multigraph $\lambda K_V$ equipped with multisets $\{M_v: v \in V\}$ of positive integers such that $\sum_{v \in V}\sigma(M_v)=\l\binom{n}{2}$, and let $f$ be a restriction function for $\mathcal{K}$ such that $\Delta_f(\mathcal{K}) < 0$. Then
\[\sum_{v \in V} (\sigma(M_v)-\sigma_{f(v)}(M_v))=\l\mbinom{n}{2}-\Delta^-_f(\mathcal{K}) < \sum_{\{u,v\} \in \binom{V}{2}} \max(\l-f(u)-f(v),0).\]
\end{lemma}

\begin{proof}
The equality follows by subtracting $\Delta^-_f(\mathcal{K})$ from each side of $\sum_{v \in V}\sigma(M_v)=\l\binom{n}{2}$ and then applying the definition of $\Delta^-_f(\mathcal{K})$. Now, because $\Delta_f(\mathcal{K}) < 0$, $\Delta^+_f(\mathcal{K}) < \Delta^-_f(\mathcal{K})$ and hence $\l\binom{n}{2}-\Delta^-_f(\mathcal{K})<\l\binom{n}{2}-\Delta^+_f(\mathcal{K})$. Then the inequality follows using $\l\binom{n}{2} = \sum_{\{u,v\} \in \binom{V}{2}} \l$ and the definition of $\Delta^+_f(\mathcal{K})$.
\end{proof}

\section{Proof of Theorem~\ref{Theorem:NPComplete} when \texorpdfstring{$\lambda$}{lambda} is odd}\label{Section:LambdaOdd}

We begin with a result which guarantees the existence of a minimal restriction function with certain properties.

\begin{lemma}\label{Lemma:MinimalFProperties}
Let $n$ and $\lambda$ be integers such that $\lambda \geq 3$ is odd, let $\ell=\frac{\l-1}{2}$ and let $V$ be a set of $n$ vertices. Let $\mathcal{K}$ be the multigraph $\lambda K_V$ equipped with multisets $\{M_v:v\in V\}$ of integers from $\{1,\ldots,n-1\}$ such that $\sum_{v \in V}\sigma(M_v)=\l\binom{n}{2}$.
\begin{itemize}
    \item[\textup{(a)}]
If $|M_v| \leq \ell+2$ for each $v \in V$, then there is a minimal restriction function $f$ for $\mathcal{K}$ such that, for each $v \in V$,
\begin{itemize}
    \item[\textup{(i)}]
$f(v) \in \left\{
  \begin{array}{ll}
    \{0,\ell+1,\ell+2\} & \hbox{if $\sigma_{\ell}(M_v) \leq \ell(n-1)-|f^{-1}(\ell+2)|$;} \\
    \{\ell,\ell+1,\ell+2\} & \hbox{if $\sigma_{\ell}(M_v) > \ell(n-1)-|f^{-1}(\ell+2)|$;}
  \end{array}
\right.$
    \item[\textup{(ii)}]
$f(v) \neq \ell+1$ if $|M_v|=\ell+2$ and $\min(M_v)=\lfloor\frac{1}{2}(\sigma(M_v)-\sigma_{\ell}(M_v))\rfloor$.
\end{itemize}
\item[\textup{(b)}]
If there are positive integers $k \geq \ell+1$ and $m \leq n-1$ such that, for each $v \in V$, $|M_v| \leq k+2$ and $\sigma_{i}(M_v)=i m$ for $i \in \{0,\ldots,k\}$, then there is a minimal restriction function $f$ for $\mathcal{K}$ such that $f(v)=0$ or $f(v) \geq \ell+2$ for each $v  \in V$.
\end{itemize}
\end{lemma}

\begin{proof}
Of all the minimal restriction functions for $\mathcal{K}$, let $f$ be one such that $|f^{-1}(0)|$ is maximised and, subject to this, $|f^{-1}(\ell+1)|$ is minimised. For brevity, let $n_i=|f^{-1}(i)|$ for each nonnegative integer $i$.

We first prove (a). Suppose in accordance with (a) that $|M_v| \leq \ell+2$ for each $v \in V$. In view of our definition of $f$, (a)(ii) follows immediately from parts (a) and (b) of Lemma~\ref{Lemma:RestrictionRestriction}. Suppose there is a vertex $w \in V$ such that $f(w) \in \{0,\ldots,\ell\}$. For each $i \in \{0,\ldots,\ell\}$, let $f_i$ be the restriction function for $\mathcal{K}$ defined by $f_i(v)=f(v)$ for $v \in V \setminus \{w\}$ and $f_i(w)=i$.
Then $\Delta^-_{f_i}(\mathcal{K})=\Delta^-_{f_0}(\mathcal{K})+\sigma_i(M_w)$. Because $f(v) \leq |M_v| \leq \ell+2$ for each $v \in V$, $\Delta^+_{f_i}(\mathcal{K})=\Delta^+_{f_0}(\mathcal{K})+i(n-1)$ if $i \in \{0,\ldots,\ell-1\}$ and $\Delta^+_{f_\ell}(\mathcal{K})=\Delta^+_{f_0}(\mathcal{K})+\ell(n-1)-n_{\ell+2}$. Thus,
\[\Delta_{f_i}(\mathcal{K}) = \left\{
  \begin{array}{ll}
    \Delta_{f_0}(\mathcal{K}) + i(n-1) - \sigma_i(M_w) & \hbox{if $i \in \{0,\ldots,\ell-1\}$;} \\
    \Delta_{f_0}(\mathcal{K}) + \ell(n-1) - n_{\ell+2} - \sigma_\ell(M_w) & \hbox{if $i=\ell$.}
  \end{array}
\right.\]
So, for each $i \in \{1,\ldots,\ell-1\}$, $\Delta_{f_0}(\mathcal{K}) \leq \Delta_{f_i}(\mathcal{K})$ because $\sigma_{i}(M_w) \leq i(n-1)$. Furthermore, if $\sigma_{\ell}(M_w) > \ell(n-1)-n_{\ell+2}$, then $\Delta_{f_\ell}(\mathcal{K}) < \Delta_{f_0}(\mathcal{K})$, and if $\sigma_{\ell}(M_w) \leq \ell(n-1)-n_{\ell+2}$, then $\Delta_{f_0}(\mathcal{K}) \leq \Delta_{f_\ell}(\mathcal{K})$. In view of our definition of $f$, this establishes (a)(i).

We now prove (b). Suppose in accordance with (b) that there are positive integers $k \geq \ell+1$ and $m \leq n-1$ such that, for each $v \in V$, $|M_v| \leq k+2$ and $\sigma_{i}(M)=i m$ for $i \in \{0,\ldots,k\}$. Note that $f(v) \in \{0,k,k+1,k+2\}$ for each $v \in V$ by Lemma~\ref{Lemma:RestrictionRestriction}(a) in view of our definition of $f$. So if $k \geq \ell+2$, then the result follows immediately and it suffices to assume $k = \ell+1$ and show that $f(v) \neq \ell+1$ for each $v \in V$. Note that we have $f(w_1)=0$ for some $w_1 \in V$, for otherwise we would have $f(v) \geq \ell+1$ for each $v \in V$ and could conclude successively that $\Delta^+_f(\mathcal{K})=\l\binom{n}{2}$, that $\Delta_f(\mathcal{K})=0$, and the contradiction that $f$ is uniformly 0.

Let $w \in V$ be a vertex such that $f(w) \in \{0,\ell+1\}$. For each $i \in \{0,\ell+1\}$, let $f_i$ be the restriction function for $\mathcal{K}$ defined by $f_i(v)=f(v)$ for $v \in V \setminus \{w\}$ and $f_i(w)=i$. Then $\Delta^-_{f_{\ell+1}}(\mathcal{K}) = \Delta^-_{f_0}(\mathcal{K})+\sigma_{\ell+1}(M_w) = \Delta^-_{f_0}(\mathcal{K})+(\ell+1)m$ and, because $f(v) \leq |M_v| \leq \ell+3$ for each $v \in V$, $\Delta^+_{f_{\ell+1}}(\mathcal{K}) = \Delta^+_{f_0}(\mathcal{K})+(\ell+1)(n-1)-d+\delta$ where $d=\sum_{j\in \{1,2,3\}}\min(\ell+1,j)n_{\ell+j}$, $\delta=1$ if $f(w)=\ell+1$, and $\delta=0$ if $f(w)=0$ (note that $\sum_{j\in \{1,2,3\}}\min(\ell+1,j)|f_0^{-1}(\ell+j)|=d-\delta$). Thus,
\[\Delta_{f_{\ell+1}}(\mathcal{K}) = \Delta_{f_0}(\mathcal{K}) + (\ell+1)(n-m-1) - d + \delta.\]

Setting $w=w_1$, we see  that $(\ell+1) (n-m-1) \geq d$ because in this case we must have \mbox{$\Delta_{f_0}(\mathcal{K}) \leq \Delta_{f_{\ell+1}}(\mathcal{K})$} by our definition of $f$. Thus, if there were any vertex $w_2$ such that \mbox{$f(w_2)=\ell+1$}, we would obtain a contradiction to our definition of $f$ by setting $w=w_2$. So $f(v) \neq \ell+1$ for each $v \in V$.
\end{proof}

Our next result will allow us to accomplish the reduction of \textsc{3-partition} to \textsc{$(\lambda,\alpha)$-star decomp} in the case where $\lambda \geq 3$ is odd and $\alpha > \frac{\lambda}{\lambda+1}$.

\begin{samepage}
\begin{lemma}\label{Lemma:LambdaOddReduction}
Let $\lambda\geq 3$ be an odd integer, let $\{a_1,\ldots,a_{3q}\}$ be an instance of \textsc{3-partition}, let $\ell=\frac{\l-1}{2}$, and let $a=\frac{1}{q}(a_1+\cdots+a_{3q})$. Suppose that $n > 4(\ell+4)(a+1)q$ is an integer such that $n \equiv q+1 \mod{\l+1}$ and let $m=\frac{\lambda(n-1)+q}{\lambda+1}$ be an integer and $b= \ell(n-m-1)+\frac{q-1}{2}-(\ell+1)qa$. Let $B=\{b^{[q]}\}$ if $q$ is odd and $B=\{\lceil b \rceil^{[q/2]},\lfloor b \rfloor^{[q/2]}\}$ if $q$ is even, and let $M=\{m^{[(\ell+1)n-q]},(\ell+1)qa_1,\ldots,(\ell+1)qa_{3q}\} \uplus B$. There is a decomposition of $\lambda K_n$ into stars of sizes given by the elements of $M$ if and only if $\{a_1,\ldots,a_{3q}\}$ is a feasible instance of \textsc{3-partition}.
\end{lemma}
\end{samepage}

\begin{proof}
Broadly, our proof strategy is to show that a decomposition of $\lambda K_n$ into stars of sizes given by $M$ exists if and only if $n-q$ vertices each have stars of sizes $\{m^{[\ell+1]}\}$ centred at them and the remaining $q$ vertices each have stars of sizes $\{m^{[\ell]},b_v\} \uplus A_v$ centred at them, where $b_v \in \{\lceil b \rceil,\lfloor b \rfloor\}$ and $A_v$ is a subset of $\{(\ell+1)qa_1,\ldots,(\ell+1)qa_{3q}\}$ with $\sigma(A_v)=(\ell+1)qa$. The values of $m$ and $b$ and the multiset $M$ have been carefully chosen to ensure that this is the case. It is then not too hard to show that $\{(\ell+1)qa_1,\ldots,(\ell+1)qa_{3q}\}$ can be partitioned into $q$ such sets $A_v$ if and only if $\{a_1,\ldots,a_{3q}\}$ is a feasible instance of \textsc{3-partition}.

Observe that $b$ is an integer if $q$ is odd and $b$ is an odd multiple of $\frac{1}{2}$ if $q$ is even. Also,
\[\sigma(M)=m((\ell+1)n-q)+(\ell+1)q^2a+bq=\lambda\mbinom{n}{2},\]
where the second equality is obtained by applying the definitions of $b$ and $m$ and using $\lambda=2\ell+1$. It will be useful to note the following facts.
\begin{align}
m &> \lceil b \rceil + (\ell+1)qa + (\ell+\tfrac{1}{2})(q-1) \label{Equation:OddLambdaFact1}\\
2\lfloor b \rfloor &> b+(\ell+1)qa + (\ell+\tfrac{1}{2})(q-1) \label{Equation:OddLambdaFact2}
\end{align}
Using $\lceil b \rceil \leq b+\frac{1}{2}$ and the definition of $b$, we see that to establish \eqref{Equation:OddLambdaFact1} it suffices to show that $(\ell+1)m>\ell(n-1)+(\ell+1)(q-1)+\frac{1}{2}$. This does indeed hold because $(\ell+1)m=(\ell+\frac{1}{2})(n-1)+\frac{q}{2}$ using the definition of $m$ and $\l=2\ell+1$, and because our hypothesis on $n$ implies that $\frac{1}{2}(n-1)>(\ell+1)(q-1)+\frac{1}{2}$. Using $\lfloor b \rfloor \geq b-\frac{1}{2}$ and the definition of $b$, we see that to establish \eqref{Equation:OddLambdaFact2} it suffices to show that $\ell(n-m-1)>2(\ell+1)qa+\ell(q-1)+1$. This does indeed hold because $\ell(n-m-1)=\frac{\ell(n-q-1)}{2\ell+2}$ using the definition of $m$ and $\l=2\ell+1$, and because our hypothesis on $n$ implies that $\ell(n-q-1)>4(\ell+1)^2q(a+1)$. Let $V$ be a set of $n$ vertices.

\textbf{`If' direction.} Suppose that $\{a_1,\ldots,a_{3q}\}$ is a feasible instance of \textsc{3-partition}. Then clearly there is a partition $\{A_1,\ldots,A_q\}$ of $\{(\ell+1)qa_1,\ldots,(\ell+1)qa_{3q}\}$ such that $\sigma(A_i)=(\ell+1)qa$ for each $i \in \{1,\ldots,q\}$. Let $V'$ be a set of $q$ vertices in $V$, let $V''=V \setminus V'$, and let $V^*$ be a subset of $V'$ such that $|V^*|=0$ if $q$ is odd and $|V^*|=\frac{q}{2}$ if $q$ is even. By Lemma~\ref{Lemma:bloodyObvious} it suffices to show that there is a star $\mathcal{K}$-decomposition where $\mathcal{K}$ is the multigraph $\l K_V$ equipped with multisets $M_v=\{m^{[\ell+1]}\}$ for each $v \in V''$, $M_v=\{m^{[\ell]},\lfloor b \rfloor +(\ell+1)qa\}$ for each $v \in V' \setminus V^*$ and $M_v=\{m^{[\ell]},\lceil b \rceil +(\ell+1)qa\}$ for each $v \in V^*$ (recall that $\lceil b \rceil +(\ell+1)qa < m$). Let $f$ be a minimal restriction function for $\mathcal{K}$ given by Lemma~\ref{Lemma:MinimalFProperties}(a). By Theorem~\ref{Theorem:RealTruthLambda} it suffices to show that $\Delta_f(\mathcal{K}) \geq 0$.

Since $|M_v| = \ell+1$ and $\sigma_{\ell}(M_v)=\ell m$ for each $v \in V$, it follows from Lemma~\ref{Lemma:MinimalFProperties}(a)(i) that $f(v) \in \{0,\ell+1\}$ for each $v \in V$. Let $V_i=f^{-1}(i)$ and $n_i=|V_i|$ for $i \in \{0,\ell+1\}$. We may assume $n_{\ell+1} < n$, for otherwise $\Delta^+_f(\mathcal{K})=\l\binom{n}{2}$ and hence $\Delta_f(\mathcal{K}) \geq 0$. We claim that
\begin{align*}
\Delta_f(\mathcal{K})&=\lambda\mbinom{n_{\ell+1}}{2}+(\ell+1)n_0n_{\ell+1}-\sum_{v \in V_{\ell+1}}\sigma(M_v) \\
&= \tfrac{1}{2}n_{\ell+1}(2\ell(n-1) + 2n - n_{\ell+1} -1)-\sum_{v \in V_{\ell+1}}\sigma(M_v)\\
&\geq \left\{
  \begin{array}{ll}
    \frac{1}{2}n_{\ell+1}(n-n_{\ell+1}-q) & \hbox{if $n_{\ell+1} \leq n-q$;} \\
    \frac{1}{2}(n-n_{\ell+1}-1)(n_{\ell+1}-(n-q)) & \hbox{if $n-q < n_{\ell+1} \leq n-1$.}
  \end{array}
\right.
\end{align*}
If this claim holds then $\Delta_f(\mathcal{K}) \geq 0$ as required, so it suffices to prove the claim. The second equality of the claim can be obtained by substituting $n_0=n-n_{\ell+1}$ and $\l=2\ell+1$. To see that the inequality holds note that $|V'|=q$, $\sigma(M_v) \leq (\ell+1)m$ for each $v \in V$, and $\sigma(M_v) \leq \ell m + b + (\ell+1)qa + \frac{1}{2}=\ell(n-1)+\frac{q}{2}$ for each $v \in V'$, and hence that
\[\sum_{v \in V_{\ell+1}}\sigma(M_v) \leq
\left\{
  \begin{array}{ll}
    n_{\ell+1}(\ell+1)m & \hbox{if $n_{\ell+1} \leq n-q$;} \\
    (n-q)(\ell+1)m+(n_{\ell+1}-(n-q))(\ell(n-1)+\frac{q}{2}) & \hbox{if $n-q < n_{\ell+1} \leq n-1$.}
  \end{array}\right.\]
Using this fact, the inequality in the claim can be established by routine calculation after applying the definition of $m$ and using $\l=2\ell+1$.

\textbf{`Only if' direction.} We do not retain any of the notation defined in the proof of the `if' direction. Suppose there is a star $\mathcal{K}$-decomposition, where $\mathcal{K}$ is $\lambda K_V$ equipped with some multisets $\{M_v:v \in V\}$ such that $\biguplus_{v \in V}M_v=M$.

Let $r=(\ell+1)n-q$. Let $\{V',V''\}$ be a partition of $V$ such that $|V'|=q$ and $\nu_m(M') \leq \ell q$ where $M'=\biguplus_{v \in V'} M_v$. Such a partition exists by pigeonhole arguments because $\nu_m(M)=r=(\ell+1)n-q$. Let $M''=\biguplus_{v \in V''} M_v$.

We will show that $M''= \{m^{[r-\ell q]}\}$. Note that, by the definitions of $M$ and $\{V',V''\}$, we have $\{m^{[r-\ell q]}\} \subseteq M''$. By Lemma~\ref{Lemma:partitionV}(a),
\begin{equation}\label{Equation:OddLambdaPartition1}
\sum_{x \in M''}(x-q) \leq \lambda\mbinom{n-q}{2}= (\ell + 1)(n - q)(m - q) = (r-\ell q)(m-q),
\end{equation}
where the first equality follows from our definition of $m$ and the second from our definition of $r$. So, because $\{m^{[r-\ell q]}\} \subseteq M''$ and each element of $M$ is greater than $q$, it must be that $M''=\{m^{[r-\ell q]}\}$ and the inequality in \eqref{Equation:OddLambdaPartition1} can be replaced with an equality.

Thus we can apply Lemma~\ref{Lemma:partitionV}(b) to obtain, for each $u \in V'$,
\begin{equation}\label{Equation:LambdaOddReductionOneSmall1}
\sigma(M_u) \leq \lambda(n-1)-|M''| = \ell(n + q -2)+q-1=\ell m + b + (\ell+1) qa  + (\ell+\tfrac{1}{2})(q-1),
\end{equation}
where the first equality follows using $|M''|=r-\ell q$ and the definition of $r$ and the second equality follows using the definitions of $b$ and $m$. Because $M''=\{m^{[r-\ell q]}\}$, we have $M'=\{m^{[\ell q]},(\ell+1)qa_1,\ldots,(\ell+1)qa_{3q}\} \uplus B$. Thus, using \eqref{Equation:LambdaOddReductionOneSmall1}, \eqref{Equation:OddLambdaFact1} and \eqref{Equation:OddLambdaFact2}, we can conclude successively that $\nu_m(M_v)=\ell$ and $\nu_{\{\lfloor b \rfloor,\lceil b \rceil\}}(M_v)=1$ for each $v \in V'$. It follows that $M_v=\{m^{[\ell]},b_v\} \uplus A_v$ for each $v \in V'$ where $\{b_v:v\in V\}=B$ and $\{A_v:v \in V'\}$ is a partition of $\{(\ell+1)qa_1,\ldots,(\ell+1)qa_{3q}\}$ into $q$ classes. So, because $\sigma(A_v) \equiv 0 \mod{(\ell+1)q}$ for each $v \in V'$ and  \mbox{$(\ell+1)q > (\ell+\tfrac{1}{2})(q-1) + \tfrac{1}{2}$}, it follows from \eqref{Equation:LambdaOddReductionOneSmall1} that $\sigma(A_v) = (\ell+1)qa$ for each $v \in V'$. The existence of $\{A_v:v \in V'\}$ implies there is a partition of $\{a_1,\ldots,a_{3q}\}$ into $q$ classes such that the elements of each class sum to $a$.
\end{proof}

In this $\lambda$ odd case we will use a greedy method to assign the specified star sizes to the vertices of the multigraph when proving that every instance of \textsc{$(\lambda,\alpha)$-star decomp} is feasible for $\alpha \leq \frac{\l}{\l+1}$. We now detail this method, and prove some basic properties of the assignment it produces. Let $m_1,\ldots,m_t$ be positive integers and $V$ be a set of $n$ vertices. A \emph{greedy assignment of $m_1,\ldots,m_t$ to multisets $\{M_v:v\in V\}$} is one produced according to the following iterative procedure. At each stage, take a largest unassigned element of $\{m_1,\ldots,m_t\}$ and assign it to a multiset $M_u$ such that the sum of the elements already assigned to $M_u$ is at most the sum of the elements already assigned to $M_v$ for each $v \in V$. Continue until all elements of $\{m_1,\ldots,m_t\}$ are assigned.

\begin{lemma}\label{Lemma:GreedyAllocationProperties}
Let $\l$ be a fixed odd integer. Let $n$, $m$ and $m_1,\ldots,m_t$ be positive integers such that $n > m \geq m_1 \geq \cdots \geq m_t$, $m_{t-1}+m_t > m > \frac{n}{2}$. Let $V$ be an index set of cardinality $n$. A greedy assignment $\{M_v:v \in V\}$ of $m_1,\ldots,m_t$ can be produced in polynomial time in $n$ and for any such assignment the following hold.
\begin{itemize}
    \item[\textup{(G1)}]
For each $v \in V$, $|M_v| \in \{\lfloor\frac{t}{n}\rfloor,\lceil\frac{t}{n}\rceil\}$.
    \item[\textup{(G2)}]
For any $u,v \in V$, $\sigma(M_v) \leq \sigma(M_u)+\min(M_v)$.
    \item[\textup{(G3)}]
For any $u,v \in V$ such that $|M_v| \geq 1$, $\sigma(M_u) \geq \frac{|M_v|-1}{|M_v|}\sigma(M_v)$.
    \item[\textup{(G4)}]
For any $u,v \in V$ such that $|M_v|=|M_u|$, $\sigma(M_v) \leq \sigma(M_u)+\max(M_v)-\min(M_u)$.
    \item[\textup{(G5)}]
For any $u,v \in V$ such that $|M_v|=|M_u|+1$, $\sigma(M_v) > \sigma(M_u)$.
\end{itemize}
\end{lemma}

\begin{proof}
Producing a greedy assignment clearly takes only polynomial time in $n$. To show that (G1)--(G5) hold, we proceed by induction on $t$. The result is obvious for $t=1$, so suppose it is true for $t \in\{1,\ldots,t'\}$ for some positive integer $t'$. We must show it also holds when integers $m_1 \geq \ldots \geq m_{t'+1}$ are assigned. Let $\{M'_v:v \in V\}$ be the multisets resulting from assigning $m_1,\ldots,m_{t'}$ and $\{M_v:v \in V\}$ be the multisets resulting from assigning $m_1,\ldots,m_{t'+1}$. We now establish (G1), (G2), (G3), (G4) and (G5) hold for $\{M_v:v \in V\}$.
\begin{itemize}
    \item[\textup{(G1)}]
Because (G1) holds for $t=t'$, we have $|M'_v| \in \{\lfloor\frac{t'}{n}\rfloor,\lceil\frac{t'}{n}\rceil\}$ for each $v \in V$. Let $M_w$ be the multiset to which the $(t'+1)$st integer is assigned. Because (G5) holds for $t=t'$, we have $\sigma(M'_v) > \sigma(M'_u)$ for any $u,v \in V$ such that $|M'_v|=|M'_u|+1$. So, because $\sigma(M'_w) \leq \sigma(M'_v)$ for each $v \in V$, we have $|M'_w|=\lfloor\frac{t'}{n}\rfloor$ and it follows that $|M_v| \in \{\lfloor\frac{t'+1}{n}\rfloor,\lceil\frac{t'+1}{n}\rceil\}$ for each $v \in V$.
    \item[\textup{(G2)}]
Suppose for a contradiction that $\sigma(M_v)-\min(M_v) > \sigma(M_u)$. Then, when the last integer was assigned to $M_v$, the sum of the integers already assigned to $M_v$ was greater than the sum of the integers already assigned to $M_u$ contradicting our greedy assignment method.
    \item[\textup{(G3)}]
We have $\sigma(M_v) \leq \sigma(M_u)+\min(M_v)$ by (G2). Thus, because $\min(M_v) \leq \frac{1}{|M_v|}\sigma(M_v)$, the result follows.
    \item[\textup{(G4)}]
Because (G1) holds for each $t \in \{1,\ldots,t'+1\}$ we have that, for each $i \in \{1,\ldots,|M_v|-1\}$, the $(i+1)$st integer assigned to $M_v$ is less than or equal to the $i$th integer assigned to $M_u$ (because it was assigned later). Thus $\sigma(M_v)-\max(M_v) \leq \sigma(M_u)-\min(M_u)$.
    \item[\textup{(G5)}]
Similarly, for each $i \in \{1,\ldots,|M_u|-1\}$, the $(i+1)$st integer assigned to $M_u$ is less than or equal to the $i$th integer assigned to $M_v$. Thus $\sigma(M_u)-\max(M_u) \leq \sigma_{|M_v|-2}(M_v)$. We have $\sigma(M_v) > \sigma_{|M_v|-2}(M_v)+m$ because $m_{t'}+m_{t'+1}>m$, and we have $\max(M_u) \leq m$ because $m_1 \leq m$. Thus $\sigma(M_u) < \sigma(M_v)$.\qedhere
\end{itemize}
\end{proof}

Our last lemma for this section shows that, when $\alpha \leq \frac{\lambda}{\lambda+1}$, every instance of \textsc{$(\lambda,\alpha)$-star decomp} is feasible. We show this by first greedily assigning (as in Lemma~\ref{Lemma:GreedyAllocationProperties}) the specified star sizes to the vertices of $\l K_n$. We then ``compress'' the resulting list at each vertex so as to reduce the number of possible restriction functions we need to consider. Finally we use Theorem~\ref{Theorem:RealTruthLambda} and Lemma~\ref{Lemma:MinimalFProperties} to establish the existence of the decomposition into the compressed sizes, and hence also of the desired decomposition. To achieve this we are forced to consider a number of cases.

\begin{lemma}\label{Lemma:LambdaOddExistence}
Let $\lambda \geq 3$ be an odd integer. For any positive integer $n$, if $M$ is a multiset of positive integers such that $\sigma(M)=\l\binom{n}{2}$ and \mbox{$\max(M) \leq  \frac{\lambda(n-1)+1}{\lambda+1}$}, then a decomposition of $\lambda K_n$ into stars of sizes given by the elements of $M$ exists and can be found in polynomial time in $n$.
\end{lemma}

\begin{proof}
The result is obvious for $n \leq 4$, so we may assume that $n \geq 5$. Let $m=\lfloor\frac{\lambda(n-1)+1}{\lambda+1}\rfloor$ and $\ell=\frac{\l-1}{2}$.
By Lemma~\ref{Lemma:bloodyObvious} we may assume that $x+y>m$ for any distinct (but possibly equal) $x,y \in M$. Let $V$ be a set of $n$ vertices. By Lemma~\ref{Lemma:GreedyAllocationProperties}, a greedy assignment $\{M_v:v \in V\}$ of the elements of $M$ to multisets can be produced in polynomial time. We will first establish that $\sigma(M_v) > \ell m$ for each $v\in V$. If, to the contrary, $\sigma(M_w) \leq \ell m$ for some $w \in V$, then $\sigma(M_v) \leq (\ell+1)m$ for each $v\in V \setminus \{w\}$ by (G2). So $\sigma(M) \leq (n-1)(\ell+1)m+\ell m$ and, using $m \leq \frac{\l(n-1)+1}{\l+1}$, it can be seen that $\l\binom{n}{2}-\sigma(M) \geq \frac{\l-1}{2(\l+1)}(n-2)>0$, contradicting our hypotheses.

In each of two cases below we will define, for each $v \in V$, a ``compressed'' multiset $M^*_v$ of integers from $\{1,\ldots,m\}$ such that $\sigma(M^*_v)=\sigma(M_v)$ and $\sigma_i(M^*_v) \geq \sigma_i(M_v)$ for each $i \in \{1,\ldots,|M^*_v|\}$. As discussed in Section~\ref{Section:Strategy}, by Lemma~\ref{Lemma:multistarDecomp} it will suffice to find a star $\mathcal{K}$-decomposition where $\mathcal{K}$ is the multigraph $\l K_V$ equipped with the multisets $\{M^*_v:v \in V\}$.

For each case we will define $f$ to be a minimal restriction function for $\mathcal{K}$ given by Lemma~\ref{Lemma:MinimalFProperties}. For each nonnegative integer $i$, let $V_i=f^{-1}(i)$ and $n_i=|V_i|$. By Theorem~\ref{Theorem:RealTruthLambda}, it will suffice to show $\Delta_f(\mathcal{K}) \geq 0$ (note that Theorem~\ref{Theorem:RealTruthLambda} guarantees a polynomial time construction). Suppose for a contradiction that $\Delta_f(\mathcal{K}) <0$. In each case below we will obtain the required contradiction by applying Lemma~\ref{Lemma:dualThing} and obtaining an upper bound for $\Delta_f^-(\mathcal{K})$. Note that $f(v) \leq \ell$ for some $v \in V$, because otherwise $\Delta^+_f(\mathcal{K})=\l\binom{n}{2}$ contradicting $\Delta_f(\mathcal{K})<0$. Because $\sigma(M^*_v) = \sigma(M_v)$ for each $v \in V$, we will use $\sigma(M_v)$ in preference to $\sigma(M^*_v)$ for the sake of clean notation.

\textbf{Case 1.} Suppose that $\sigma(M_{v}) > (\ell+2)m$ for some $v \in V$. By (G2), we must have $km < \sigma(M_v) \leq (k+2)m$ for each $v \in V$ and some $k \geq \ell+1$. For each $v \in V$, let
\[M^*_v=\left\{
  \begin{array}{ll}
    \{m^{[k]},\sigma(M_v)-k m\} & \hbox{if $k m < \sigma(M_v) \leq (k+1) m$;} \\
    \{m^{[k+1]},\sigma(M_v)-(k+1) m\} & \hbox{if $(k+1) m < \sigma(M_v) \leq (k+2) m$.}
  \end{array}
\right.\]
Note $\{m^{[k]}\} \subseteq M^*_v$ for each $v \in V$. So, by Lemma~\ref{Lemma:MinimalFProperties}(b), we can take $f$ to be a minimal restriction function for $\mathcal{K}$ such that $f(v)=0$ or $f(v) \geq \ell+2$ for each $v \in V$. Thus, by Lemma~\ref{Lemma:dualThing}, because $\sum_{v \in V_0}\sigma(M_v) \leq \sum_{v \in V}(\sigma(M_v)-\sigma_{f(v)}(M_v))$, we have
\begin{equation}\label{Equation:LambdaOddNewCase}
\sum_{v \in V_0}\sigma(M_v) \leq \l\mbinom{n}{2}-\Delta^-_f(\mathcal{K})  < \l\mbinom{n_0}{2} + (\ell-1)n_0(n-n_0).
\end{equation}
It follows that $\sigma(M_w) < \frac{\l}{2}(n_0-1)+(\ell-1)(n-n_0)$ for some $w \in V_0$ and substituting $\ell=\frac{\l-1}{2}$ shows this latter expression is equal to $\frac{\l}{2}(n-1)-\frac{3}{2}(n-n_0)$. So $\sigma(M_v) \leq \frac{\l}{2}(n-1)-\frac{3}{2}(n-n_0)+m$ for each $v \in V$ by (G2). Thus, $\Delta^-_f(\mathcal{K}) \leq (n-n_0)(\frac{\l}{2}(n-1)-\frac{3}{2}(n-n_0)+m)$. Adding this to the second and third expressions in \eqref{Equation:LambdaOddNewCase}, using $\ell=\frac{\l-1}{2}$ and $\binom{n}{2}=\binom{n_0}{2}+\frac{1}{2}(n-n_0)(n+n_0-1)$, we obtain the contradiction
\[\l\mbinom{n}{2}  < \l\mbinom{n}{2}-\tfrac{1}{2}(n-n_0)(3n-2m) \leq \l\mbinom{n}{2}.\]

\textbf{Case 2.} Suppose that $\sigma(M_v) \leq (\ell+2)m$ for each $v \in V$. Recall that $\sigma(M_v) > \ell m$ for each $v \in V$. For each $v \in V$, let $y_v=\max(\sigma_{\ell+1}(M_v)-\ell m,\lceil\frac{1}{2}(\sigma(M_v)-\ell m)\rceil)$ and
\[M^*_v=\left\{
  \begin{array}{ll}
    \{m^{[\ell]},y_v\} & \hbox{if $\sigma(M_v)=\ell m +y_v$;} \\
    \{m^{[\ell]},y_v,\sigma(M_v)-\ell m -y_v\} & \hbox{if $\sigma(M_v) > \ell m +y_v$.}
  \end{array}
\right.\]
(Intuitively, $y_v$ is the smallest integer that ensures $\sigma_{\ell+1}(M^*_v) \geq \sigma_{\ell+1}(M_v)$ and $y_v \geq \sigma(M^*_v) - \sigma_{\ell+1}(M^*_v)$.) For each $v \in V$, either $M^*_v=\{m^{[\ell]},\lceil\frac{1}{2}(\sigma(M_v)-\ell m)\rceil,\lfloor\frac{1}{2}(\sigma(M_v)-\ell m)\rfloor\}$ or $\sigma_{\ell+1}(M^*_v)=\sigma_{\ell+1}(M_v)$. So by Lemma~\ref{Lemma:MinimalFProperties} we can take $f$ to be a minimal restriction function for $\mathcal{K}$ satisfying (a)(i) and (a)(ii) of Lemma~\ref{Lemma:MinimalFProperties}. Thus, because $f$ satisfies the latter of these, $\sigma_{\ell+1}(M^*_v)=\sigma_{\ell+1}(M_v)$ for each $v \in V_{\ell+1}$ and we will use $\sigma_{\ell+1}(M_v)$ in preference to $\sigma_{\ell+1}(M^*_v)$. We consider two cases according to the value of $n_{\ell+2}$.

\textbf{Case 2a.} Suppose that $n_{\ell+2} > \ell(n-m-1)$. Then $V_\ell \cup V_{\ell+1} \cup V_{\ell+2}  = V$ by Lemma~\ref{Lemma:MinimalFProperties}(a)(i) because $\sigma_\ell(M^*_v)=\ell m$ for each $v \in V$. By Lemma~\ref{Lemma:dualThing}, again using $\sigma_\ell(M^*_v)=\ell m$ for each $v \in V$, we have
\begin{equation}\label{Equation:LambdaOddExistence1}
\sum_{v \in V_{\ell}} (\sigma(M_v)-\ell m) \leq \l\mbinom{n}{2}-\Delta^-_f(\mathcal{K}) < \mbinom{n_{\ell}}{2}.
\end{equation}
So $\sigma(M_w)-\ell m < \frac{1}{2}(n_\ell-1)$ and hence $\sigma(M_w) < \ell m+\frac{1}{2}(n_\ell-1)$ for some $w \in V_{\ell}$. Thus, by (G2), $\sigma(M_v) < (\ell+1)m+\frac{1}{2}(n_\ell-1)$ for each $v \in V$ and hence $\Delta^-_f(\mathcal{K}) \leq  n_\ell \ell m +(n-n_{\ell})((\ell+1)m+\frac{1}{2}(n_\ell-1))$. Adding this to the second and third expression in \eqref{Equation:LambdaOddExistence1},
\begin{align*}
\l\mbinom{n}{2} &< \mbinom{n_\ell}{2}+n_\ell\ell m+(n-n_\ell)((\ell+1)m+\tfrac{1}{2}(n_\ell-1))\\
&=\tfrac{1}{2}n(n_\ell-1)+\tfrac{1}{2}m\bigl(n(\l+1) - 2n_\ell\bigr)\\
&\leq \l\mbinom{n}{2}-\mfrac{\l-1}{2(\l+1)}n_\ell(n-2) \leq \l\mbinom{n}{2},
\end{align*}
where the equality follows using $\ell=\frac{\l-1}{2}$ and the second inequality follows using $m \leq \frac{\l(n-1)+1}{\l+1}$.

\textbf{Case 2b.} Suppose that $0 \leq n_{\ell+2} \leq \ell(n-m-1)$. Then $V_0 \cup V_{\ell+1} \cup V_{\ell+2}  = V$ by Lemma~\ref{Lemma:MinimalFProperties}(a)(i) because $\sigma_\ell(M^*_v)=\ell m$ for each $v \in V$. Let $w$ be an element of $V_0$ such that $\sigma(M_w) \leq \sigma(M_v)$ for each $v \in V_0$. By Lemma~\ref{Lemma:dualThing} we have
\begin{equation}\label{Equation:LambdaOddExistence2}
\sum_{v \in V_0} \sigma(M_v) \leq \l\mbinom{n}{2}-\Delta^-_f(\mathcal{K}) < \lambda\mbinom{n_{0}}{2}+\ell n_0n_{\ell+1}+(\ell-1)n_0n_{\ell+2}.
\end{equation}
So, using $n_{\ell+1}=n-n_0-n_{\ell+2}$, $\sigma(M_w) < \frac{\l}{2}(n_0-1)+\ell n_{\ell+1}+(\ell-1)n_{\ell+2}=\ell(n-1)+\frac{1}{2}(n_0-1)- n_{\ell+2}$. We will use this fact often. Also, adding $\Delta^-_f(\mathcal{K})$ to the second and third expression in \eqref{Equation:LambdaOddExistence2},
\begin{equation}\label{Equation:LambdaOddExistence3}
\l\mbinom{n}{2} < \lambda\mbinom{n_{0}}{2}+\ell n_0n_{\ell+1}+(\ell-1)n_0n_{\ell+2} + \sum_{v \in  V_{\ell+1}} \sigma_{\ell+1}(M_v)+\sum_{v \in V_{\ell+2}} \sigma(M_v).
\end{equation}
We now consider two subcases according to whether $|M_w| = \ell +1$.

\textbf{Case 2b(i).} Suppose that $|M_w| \geq \ell +2$. Then $\sigma_{\ell+1}(M_v) \leq \sigma(M_w)$ for each $v \in V_{\ell+1}$ for otherwise (G2) or (G5) would have been violated immediately after $M_w$ was assigned its $(\ell+2)$nd element. By (G3), for each $v \in V_{\ell+2}$, we have that $\sigma(M_v) \leq \frac{\ell+2}{\ell+1}\sigma(M_w)$ because $|M_v| \geq \ell+2$. Thus, from \eqref{Equation:LambdaOddExistence3},
\begin{align*}
\l\mbinom{n}{2} &< \lambda\mbinom{n_{0}}{2}+\ell n_0n_{\ell+1} + (\ell-1)n_0n_{\ell+2} + n_{\ell+1}\sigma(M_w) + \tfrac{\ell+2}{\ell+1}n_{\ell+2}\sigma(M_w) \nonumber\\
&= \mfrac{n_0}{2} \Bigl(n_0 - 2n_{\ell+2} + (\l - 1) (n - 1) - 1\Bigr) +
 \left(n - n_0 + \mfrac{2n_{\ell+2}}{\l + 1}\right)\sigma(M_w)\nonumber\\
&< \left(\mfrac{n}{2}+\mfrac{n_{\ell+2}}{\l+1}\right)\left((\l-1)(n-1)+n_0-2n_{\ell+2}-1\right) \\
&\leq \l\mbinom{n}{2} - \mfrac{n_{\ell+2}}{(\l+1)}(\ell(n+2)+2n+3n_{\ell+2}+1) \leq \l\mbinom{n}{2},
\end{align*}
where the equality is obtained using $n_{\ell+1}=n-n_0-n_{\ell+2}$ and $\ell=\frac{\l-1}{2}$, the second inequality is obtained using $\sigma(M_w) < \ell(n-1)+\frac{1}{2}(n_0-1)- n_{\ell+2}$ and $\ell=\frac{\l-1}{2}$, and the third inequality is obtained using $n_0 \leq n-n_{\ell+2}$.

\textbf{Case 2b(ii).} Suppose that $|M_w| = \ell +1$. Let $s=\min(M_w)$. Then, for each $v \in V_{\ell+2}$, $\sigma(M_v) \leq \sigma(M_w)+\min(M_v)$ by (G2) and $\min(M_v) \leq s$ because $|M_v| > |M_w|$.
Also, $\sigma_{\ell+1}(M_v) \leq \sigma(M_w)+m-s$ for each $v \in V_{\ell+1}$ using (G4).  Thus, from \eqref{Equation:LambdaOddExistence3},
\begin{align}
\l\mbinom{n}{2} &< \lambda\mbinom{n_{0}}{2}+\ell n_0n_{\ell+1} + (\ell-1)n_0n_{\ell+2} + n_{\ell+1}(\sigma(M_w)+m-s) + n_{\ell+2}(\sigma(M_w)+s) \nonumber \\
&= s(n_0+2n_{\ell+2}-n) + m(n-n_0-n_{\ell+2})+\sigma(M_w)(n-n_0)+n_0(\ell(n-1)+\tfrac{n_0-1}{2}-n_{\ell+2}), \label{Equation:LambdaOddExistence4}
\end{align}
where the equality is obtained using $n_{\ell+1}=n-n_0-n_{\ell+2}$. We will obtain a contradiction from \eqref{Equation:LambdaOddExistence4}. The sign of $n_0+2n_{\ell+2}-n$ determines whether we require an upper or lower bound for $s$.

If $n_0+2n_{\ell+2} \geq n$, then using first $s \leq \frac{1}{\ell+1}\sigma(M_w)$ and next $\sigma(M_w) < \ell(n-1)+\frac{1}{2}(n_0-1)- n_{\ell+2}$ and $m \leq \frac{\l(n-1)+1}{\l+1}$, we have from \eqref{Equation:LambdaOddExistence4} that
\begin{align*}
\l\mbinom{n}{2} &< \mfrac{1}{2(\l+1)}\left(n(\l^2(n-1)+n-\l+2)+n_0((\l-3)n+2n_0-2)-2n_{\ell+2}(n+\l+4n_{\ell+2}+1)\right)\\
&\leq \l\mbinom{n}{2}-\mfrac{n-n_0}{2(\l+1)}(\l(n+1)+2n-1) \leq \l\mbinom{n}{2},
\end{align*}
where the second inequality is obtained by recalling that $n_{\ell+2} \geq \tfrac{1}{2}(n-n_0)$ and hence expression is maximised when $n_{\ell+2} = \tfrac{1}{2}(n-n_0)$.

If $n_0+2n_{\ell+2} <n$, then using first $s \geq \sigma(M_w)-\ell m$ and next $\sigma(M_w) < \ell(n-1)+\frac{1}{2}(n_0-1)- n_{\ell+2}$, we have from \eqref{Equation:LambdaOddExistence4} that
\begin{align*}
\l\mbinom{n}{2} &<\tfrac{1}{2}m(\l(n-n_0-2n_{\ell+2})+n-n_0)+\tfrac{1}{2}(n_0+2n_{\ell+2})((\l-1)(n-1)+n_0-2n_{\ell+2}-1)\\
&\leq \l \mbinom{n}{2} - n_{\ell+2}(2n_{\ell+2} + \tfrac{n+2\l}{\l + 1})-\tfrac{1}{2}(n_0-1)(n-n_0) < \l \mbinom{n}{2},
\end{align*}
where the second inequality follows using $m \leq \frac{\l(n-1)+1}{\l+1}$ and the third follows using $n_0 \geq 1$ (recall that $f(v) \leq \ell$ for some $v \in V$).
\end{proof}

\begin{proof}[\textbf{\textup{Proof of Theorem~\ref{Theorem:NPComplete} when $\lambda$ is odd}}] If $\alpha \leq \frac{\l}{\l+1}$, Lemma~\ref{Lemma:LambdaOddExistence} shows that every instance of \textsc{$(\lambda,\alpha)$-star decomp} is feasible and that the required decompositions can be constructed in polynomial time. If $\alpha > \frac{\l}{\l+1}$ and $\{a_1,\ldots,a_{3q}\}$ is an instance of \textsc{3-partition}, then we can apply Lemma~\ref{Lemma:LambdaOddReduction}, with $n$ chosen to be polynomial in $a_1+\cdots+a_{3q}$ but sufficiently large that $m<\alpha(n-1)$, in order to reduce the instance of \textsc{3-partition} to an instance of \textsc{$(\lambda,\alpha)$-star decomp}.
\end{proof}

\section{Proof of Theorem~\ref{Theorem:NPComplete} when \texorpdfstring{$\lambda$}{lambda} is even}\label{Section:LambdaEven}

We begin with a result which guarantees the existence of a minimal restriction function with certain properties.

\begin{lemma}\label{Lemma:MinimalFProperties1LambdaEven}
Let $n$ and $\lambda$ be positive integers such that $\lambda$ is even, let $\ell=\frac{\l}{2}$, and let $V$ be a set of $n$ vertices. Let $\mathcal{K}$ be the multigraph $\lambda K_V$ equipped with multisets $\{M_v:v\in V\}$ of integers from $\{1,\ldots,n-1\}$ such that $\sum_{v \in V}\sigma(M_v)=\l\binom{n}{2}$.
\begin{itemize}
    \item[\textup{(a)}]
If $|M_v| \leq \ell+1$ for each $v \in V$, then there is a minimal restriction function $f$ for $\mathcal{K}$ such that, for each $v \in V$,
\begin{itemize}
    \item[\textup{(i)}]
$f(v) \in \left\{
  \begin{array}{ll}
    \{0,\ell+1\} & \hbox{if $\sigma_{\ell}(M_v) \leq \ell(n-1)-|f^{-1}(\ell+1)|$;} \\
    \{\ell,\ell+1\} & \hbox{if $\sigma_{\ell}(M_v) > \ell(n-1)-|f^{-1}(\ell+1)|$;}
  \end{array}
\right.$
    \item[\textup{(ii)}]
$f(v) \neq \ell$ if $|M_v|=\ell+1$ and $\min(M_v) = \lfloor\frac{1}{2}(\sigma(M_v)-\sigma_{\ell-1}(M_v))\rfloor$.
\end{itemize}
    \item[\textup{(b)}]
If there are positive integers $k \geq \ell$ and $m < n$ such that, for each $v \in V$, $|M_v| \leq k+2$ and $\sigma_{i}(M_v)=i m$ for $i \in \{0,\ldots,k\}$, then there is a minimal restriction function $f$ for $\mathcal{K}$ such that $f(v)=0$ or $f(v) \geq \ell+1$ for each $v \in V$.
\end{itemize}
\end{lemma}

\begin{proof}
Of all the minimal restriction functions for $\mathcal{K}$, let $f$ be one such that $|f^{-1}(0)|$ is maximised and, subject to this, $|f^{-1}(\ell)|$ is minimised. Let $n_i=|f^{-1}(i)|$ for each nonnegative integer $i$.

We first prove (a). Suppose in accordance with (a) that $|M_v| \leq \ell+1$ for each $v \in V$. Suppose there is a vertex $w \in V$ such that $f(w) \in \{0,\ldots,\ell\}$. For each $i \in \{0,\ldots,\ell\}$, let $f_i$ be the restriction function for $\mathcal{K}$ defined by $f_i(v)=f(v)$ for $v \in V \setminus \{w\}$ and $f_i(w)=i$. Then $\Delta^-_{f_i}(\mathcal{K})=\Delta^-_{f_0}(\mathcal{K})+\sigma_i(M_w)$ for each $i \in \{0,\ldots,\ell\}$. Because $f(v) \leq |M_v| \leq \ell+1$ for each $v \in V$, $\Delta^+_{f_i}(\mathcal{K})=\Delta^+_{f_0}(\mathcal{K})+i(n-1)$ for each $i \in \{0,\ldots,\ell-1\}$ and $\Delta^+_{f_\ell}(\mathcal{K})=\Delta^+_{f_0}(\mathcal{K})+\ell(n-1)-n_{\ell+1}$. Thus,
\[\Delta_{f_i}(\mathcal{K}) = \left\{
  \begin{array}{ll}
    \Delta_{f_0}(\mathcal{K}) + i(n-1) - \sigma_i(M_w) & \hbox{if $i \in \{0,\ldots,\ell-1\}$;} \\
    \Delta_{f_0}(\mathcal{K}) + \ell(n-1) - n_{\ell+1} - \sigma_\ell(M_w) & \hbox{if $i=\ell$.}
  \end{array}
\right.\]
So, for each $i \in \{1,\ldots,\ell-1\}$, $\Delta_{f_0}(\mathcal{K}) \leq \Delta_{f_i}(\mathcal{K})$ because $\sigma_{i}(M_w) \leq i(n-1)$. Furthermore, $\Delta_{f_\ell}(\mathcal{K}) < \Delta_{f_0}(\mathcal{K})$ if and only if $\sigma_{\ell}(M_w) > \ell(n-1)-n_{\ell+1}$. In view of our definition of $f$, this establishes (a)(i).

We now prove (a)(ii). Suppose further, for a contradiction, that $f(w) =\ell$ and that $M_w=\{m_1,\ldots,m_{\ell+1}\}$ where $m_1 \geq \cdots \geq m_{\ell+1}$ and $m_{\ell+1} = \lfloor\frac{1}{2}(\sigma(M_w)-\sigma_{\ell-1}(M_w))\rfloor$. Then $m_{\ell} = \lceil\frac{1}{2}(\sigma(M_w)-\sigma_{\ell-1}(M_w))\rceil$ and hence $m_{\ell+1} \in \{m_{\ell}-1,m_\ell\}$. So by Lemma~\ref{Lemma:RestrictionRestriction}, one of $f_{\ell-1}$ or $f_{\ell+1}$ is also a minimal restriction function for $\mathcal{K}$, contradicting our choice of $f$. This establishes (a)(ii).

Finally we prove (b). Suppose in accordance with (b) that there are positive integers $k \geq \ell$ and $m < n$ such that, for each $v \in V$, $|M_v| \leq k+2$ and $\sigma_{i}(M)=i m$ for $i \in \{0,\ldots,k\}$. Note that $f(v) \in \{0,k,k+1,k+2\}$ for each $v \in V$ by Lemma~\ref{Lemma:RestrictionRestriction}(a). So (b) follows immediately if $k \geq \ell+1$ and it suffices to assume $k = \ell$ and prove that $f(v) \neq \ell$ for each $v \in V$.

Suppose there is a vertex $w \in V$ such that $f(w) \in \{0,\ell\}$. For each $i \in \{0,\ell\}$, let $f_i$ be the restriction function for $\mathcal{K}$ defined by $f_i(v)=f(v)$ for $v \in V \setminus \{w\}$ and $f_i(w)=i$. Then $\Delta^-_{f_\ell}(\mathcal{K}) = \Delta^-_{f_0}(\mathcal{K})+\sigma_\ell(M_w) = \Delta^-_{f_0}(\mathcal{K})+\ell m$ and, because $f(v) \leq |M_v| \leq \ell+2$ for each $v \in V$, $\Delta^+_{f_\ell}(\mathcal{K}) = \Delta^+_{f_0}(\mathcal{K})+\ell(n-1)-d$, where $d=\sum_{j\in \{1,2\}}\min(\ell,j)n_{\ell+j}$. Thus,
\[\Delta_{f_\ell}(\mathcal{K}) = \Delta_{f_0}(\mathcal{K}) + \ell(n-m-1) - d.\]
 So if $\ell (n-m-1) < d$, then our definition of $f$ would imply that $f(v) \geq \ell$ for all $v \in V$ and we could conclude successively that $\Delta^+_f(\mathcal{K})=\l\binom{n}{2}$, that $\Delta_f(\mathcal{K})=0$, and the contradiction that $f$ is uniformly 0. Thus it must be that $\ell (n-m-1) \geq d$ and hence by our definition of $f$ that $f(v) \neq \ell$ for each $v \in V$.
\end{proof}

In Lemma~\ref{Lemma:LambdaEvenNPComplete} we will establish that \textsc{$(\lambda,\alpha)$-star decomp} is $\mathsf{NP}$-complete for $\lambda$ even and  $\alpha>1-\tfrac{2}{\l}(3-2\sqrt{2})$. The bulk of this work is accomplished in Lemma~\ref{Lemma:LambdaEvenReduction} which allows us to reduce an instance of \textsc{3-partition} to an instance of \textsc{$(\lambda,\alpha)$-star decomp} provided we can find suitable integers $n$, $m$ and $r$.

\begin{lemma}\label{Lemma:LambdaEvenReduction}
Let $\lambda$ be a positive even integer, let $\{a_1,\ldots,a_{3q}\}$ be an instance of \textsc{3-partition}, let $\ell=\frac{\l}{2}$, and let $a=\frac{1}{q}(a_1+\cdots+a_{3q})$. Suppose there are positive integers $n$, $m$ and $r$ such that $\ell q < r < n-q$, $q < m < n-1$ and
\begin{align*}
  c &= \mfrac{\ell(n-q)(n-q-1)-r(m-q)}{(\ell+1)(n-q)-2r}+q \quad\mbox{and} \\[1mm]
  b &= (\ell-1)(n-c-1)+r+q-c-1-(\ell+1)qa
\end{align*}
are integers satisfying $2c>m+q$, $\ell(c-q) > (\ell-1)(m-q)$, $m>c+b+(\ell+1)qa$ and $b>(\ell+1)qa+\ell(q-1)$. Then there is a decomposition of $\lambda K_n$ into stars of sizes $\{m^{[r]},c^{[(\ell+1)n-2r-q]},b^{[q]},(\ell+1)qa_1,\ldots,(\ell+1)qa_{3q}\}$ if and only if $\{a_1,\ldots,a_{3q}\}$ is a feasible instance of \textsc{3-partition}.
\end{lemma}

\begin{proof}
Let $M$ be the multiset $\{m^{[r]},c^{[(\ell+1)n-2r-q]},b^{[q]},(\ell+1)qa_1,\ldots,(\ell+1)qa_{3q}\}$ and let $V$ be a set of $n$ vertices. Note that
\[\sigma(M)=mr+ c((\ell+1)n-2r-q)+bq+(\ell+1)q^2a=\lambda\mbinom{n}{2},\]
where the second equality follows by first applying the definition of $b$ and then applying the definition of $c$.

Our proof strategy is similar to the one we employed in the proof of Lemma~\ref{Lemma:LambdaOddReduction}. We will show that a decomposition of $\lambda K_n$ into stars of sizes given by $M$ exists if and only if $n-r-q$ vertices each have stars of sizes $\{c^{[\ell+1]}\}$ centred at them, $r$ vertices each have stars of sizes $\{m,c^{[\ell-1]}\}$ centred at them, and the remaining $q$ vertices each have stars of sizes $\{c^{[\ell]},b\} \uplus A_v$ centred at them, where $A_v$ is a subset of $\{(\ell+1)qa_1,\ldots,(\ell+1)qa_{3q}\}$ with $\sigma(A_v)=(\ell+1)qa$. The values of $m$, $c$ and $b$ and the multiset $M$ have been carefully chosen to ensure that this is the case. It is then not too hard to show that $\{(\ell+1)qa_1,\ldots,(\ell+1)qa_{3q}\}$ can be partitioned into $q$ such sets $A_v$ if and only if $\{a_1,\ldots,a_{3q}\}$ is a feasible instance of \textsc{3-partition}.

\textbf{`If' direction.} Suppose that $\{a_1,\ldots,a_{3q}\}$ is a feasible instance of \textsc{3-partition}. Then clearly there is a partition $\{A_1,\ldots,A_q\}$ of $\{(\ell+1)qa_1,\ldots,(\ell+1)qa_{3q}\}$ such that $\sigma(A_i)=(\ell+1)qa$ for each $i \in \{1,\ldots,q\}$. Let $\{V',V'',V'''\}$ be a partition of $V$ such that $|V'|=q$, $|V''|=r$ and $|V'''|=n-r-q$. By Lemma~\ref{Lemma:bloodyObvious}, it will suffice to show that there is a star $\mathcal{K}$-decomposition where $\mathcal{K}$ is the multigraph $\l K_V$ equipped with multisets $\{M_v:v \in V\}$ such that $M_v=\{c^{[\ell+1]}\}$ for $v \in V'''$, $M_v=\{m,c^{[\ell-1]}\}$ for $v \in V''$, $M_v=\{c^{[\ell-1]},c+b+(\ell+1)qa\}$ for $v \in V'$ (note that $c+b+(\ell+1)qa<m$ by our hypotheses). Let $f$ be a minimal restriction function for $\mathcal{K}$ given by Lemma~\ref{Lemma:MinimalFProperties1LambdaEven}(a) and let $V_i=f^{-1}(i)$ and $n_i=|V_i|$ for $i \in \{0,\ell,\ell+1\}$. By Theorem~\ref{Theorem:RealTruthLambda} it suffices to show that $\Delta_f(\mathcal{K}) \geq 0$. We may assume that $n_0 \geq 1$, for otherwise $\Delta^+_f(\mathcal{K})=\l\binom{n}{2}$ and hence $\Delta_f(\mathcal{K}) \geq 0$.

Now $\Delta^+_f(\mathcal{K})=\l\binom{n-n_0}{2}+\ell n_0n_\ell+(\ell+1) n_0n_{\ell+1}$. Note $V_{\ell+1} \subseteq V'''$ because $|M_v| = \ell$ for $v \in V' \cup V''$. By Lemma~\ref{Lemma:MinimalFProperties1LambdaEven}(a)(ii) it can be seen that $V_\ell \cap V''' = \emptyset$.
Further, $V' \subseteq V_0$ by Lemma~\ref{Lemma:MinimalFProperties1LambdaEven}(a)(i), because $n_{\ell+1} \leq |V'''| = n-r-q$ and $\ell c+b+(\ell+1)qa = \ell(n-1)-(n-r-q)$ by the definition of $b$. Thus we see that $V_\ell \subseteq V''$, in addition to $V_{\ell+1} \subseteq V'''$.
So, since $\sigma(M_v)=m+(\ell-1)c$ for each $v\in V''$ and $\sigma(M_v) = (\ell+1)c$ for each $v \in V'''$, we have
\[\Delta^-_f(\mathcal{K}) = n_\ell (m+(\ell-1)c) +n_{\ell+1}(\ell+1)c.\]
From this, our expression for $\Delta^+_f(\mathcal{K})$, and the fact that $n_\ell=n-n_{0}-n_{\ell+1}$ we see that
\begin{equation}\label{Equation:LambdaEvenReductionIfDir}
\Delta_f(\mathcal{K}) = (n-n_0)(\ell(n-c-1)-m+c)+n_{\ell+1}(n_0+m-2c ).
\end{equation}
We will use \eqref{Equation:LambdaEvenReductionIfDir} to show that $\Delta_f(\mathcal{K}) \geq 0$, considering two cases according to the value of $n_{\ell+1}$.

If $n_{\ell+1} \leq \ell(n-c-1)-m+c$, then $V_\ell=\emptyset$ by Lemma~\ref{Lemma:MinimalFProperties1LambdaEven}(a)(i), so $n_{\ell+1}=n-n_0$ and from \eqref{Equation:LambdaEvenReductionIfDir} we have
\[\Delta_f(\mathcal{K})\geq (n-n_0)(\ell(n-c-1)+n_0-c).\]
Thus $\Delta_f(\mathcal{K})\geq 0$ because our assumption that $n_{\ell+1} \leq \ell(n-c-1)-m+c$, together with $n_{\ell+1}=n-n_0$, implies that $n+m-c \leq \ell(n-c-1)+n_0$ and hence that $\ell(n-c-1)+n_0-c \geq n+m-2c > 0$.

If, on the other hand, $n_{\ell+1} >\ell(n-c-1)-m+c $, then $V_\ell=V''$ by Lemma~\ref{Lemma:MinimalFProperties1LambdaEven}(a)(i), so $n_{\ell+1}=n-r-n_0$ and from \eqref{Equation:LambdaEvenReductionIfDir} we have
\[\Delta_f(\mathcal{K}) \geq (n - n_0)(\ell(n - c - 1)+n_0+r-c) - r(n + m - 2c).\]
Because $q \leq n_0 \leq n-r$, the right hand expression is minimised either when $n_0=q$ or when $n_0=n-r$. When $n_0=q$, the expression is equal to $0$ by the definition of $c$. When $n_0=n-r$, the expression is $r((\ell-1)(n-c-1)+n-m-1)$, which is nonnegative because $n-1 > m > c$.

\textbf{`Only if' direction.} We do not retain any of the notation defined in the proof of the `if' direction. Suppose there is a star $\mathcal{K}$-decomposition, where $\mathcal{K}$ is $\lambda K_V$ equipped with some multisets $\{M_v:v \in V\}$ such that $\biguplus_{v \in V}M_v=M$.

Let $\{V',V''\}$ be a partition of $V$ such that $|V'|=q$ and, for all $v \in V'$ and $u \in V''$,  either $\nu_{\{m,c\}}(M_{v}) < \nu_{\{m,c\}}(M_{u})$ or $\nu_{\{m,c\}}(M_{v}) = \nu_{\{m,c\}}(M_{u})$ and $\nu_{m}(M_{v}) \leq \nu_{m}(M_{u})$. Let $M'=\biguplus_{v \in V'}M_{v}$, $M''=\biguplus_{u \in V''}M_{u}$ and, for a multiset $S$ of positive integers, abbreviate $(m-q)\nu_m(S)+(c-q)\nu_c(S)$ to $\xi(S)$. By Lemma~\ref{Lemma:partitionV}(a) we have
\begin{equation}\label{Equation:LambdaEvenReductionL6b}
\xi(M'') \leq \sum_{x \in M''}(x-q) \leq \l\mbinom{n-q}{2}=(m-q)r+(c-q)((\ell+1)(n-q)-2r),
\end{equation}
where the equality follows by the definition of $c$.

We will show that $\xi(M_v)=\ell(c-q)$ for each $v \in V'$. Suppose otherwise. Because $\nu_m(M')=r-\nu_m(M'')$ and $\nu_c(M')=(\ell+1)n-2r-q-\nu_c(M'')$, \eqref{Equation:LambdaEvenReductionL6b} implies that $\xi(M') \geq \ell q(c-q)$. So $\xi(M_{v_0}) > \ell(c-q)$ for some $v_0 \in V'$ and hence either $\nu_{\{m,c\}}(M_{v_0}) \geq \ell+1$ or $\nu_{\{m,c\}}(M_{v_0}) = \ell$ and $\nu_{m}(M_{v_0}) \geq 1$. Then, by the definition of $\{V',V''\}$, for each $u \in V''$ either $\nu_{\{m,c\}}(M_u) \geq \ell+1$ or $\nu_{\{m,c\}}(M_u) = \ell$ and $\nu_{m}(M_u) \geq 1$. Furthermore, the latter applies for strictly fewer than $r$ vertices $u \in V''$, for otherwise we would necessarily have $\{m^{[r]},c^{[(\ell+1)(n-q)-2r]}\} \subseteq M''$ and hence $\nu_m(M')=0$ and $\nu_c(M') \leq \ell q$, contradicting $\xi(M') > \ell q(c-q)$. Thus, $\xi(M_u) \geq \xi(\{m,c^{[\ell-1]}\})$ for each $u \in V''$ and $\xi(M_u) \geq \xi(\{c^{[\ell+1]}\})$ for strictly more than $n-r-q$ vertices $u \in V''$ (note that our hypothesis $2c>m+q$ implies $\xi(\{m,c^{[\ell-1]}\}) < \xi(\{c^{[\ell+1]}\})$). Thus,
\[\xi(M'') > r\xi(\{m,c^{[\ell-1]}\})+(n-r-q)\xi(\{c^{[\ell+1]}\}) = \l\mbinom{n-q}{2},\]
where the equality is obtained by applying the definitions of $\xi$ and $c$. Hence we have a contradiction to \eqref{Equation:LambdaEvenReductionL6b} and it is indeed the case that $\xi(M_v)=\ell(c-q)$ for each $v \in V'$. Further, for each $v \in V'$, we have $\nu_m(M_v)+\nu_c(M_v) \geq \ell$ for otherwise $\nu_m(M_v)+\nu_c(M_v) \leq \ell-1$ and $\xi(M_v) \leq (\ell-1)(m-q) < \ell(c-q)$ where the second inequality is one of our hypotheses. Thus, because $m>c$, for each $v \in V'$ we in fact have $(\nu_m(M_v),\nu_c(M_v))=(0,\ell)$.

From this it follows that $(\nu_m(M''),\nu_c(M''))=(r,(\ell+1)(n-q)-2r)$. Thus, because each element of $M$ is greater than $q$, equality holds throughout \eqref{Equation:LambdaEvenReductionL6b} and $M''=\{m^{[r]},c^{[(\ell+1)(n-q)-2r]}\}$. So, by Lemma~\ref{Lemma:partitionV}(b), for each $v \in V'$,
\begin{equation}\label{Equation:LambdaEvenReductionOneSmall2}
\sigma(M_{v}) \leq \lambda(n-1)-|M''|=\ell c + b + (\ell+1)qa + \ell(q-1),
\end{equation}
where the equality follows using $|M''|=(\ell+1)(n-q)-r$ and the definition of $b$.

Because $M''=\{m^{[r]},c^{[(\ell+1)(n-q)-2r]}\}$, $M'=\{c^{[\ell q]},b^{[q]},(\ell+1)qa_1,\ldots,(\ell+1)qa_{3q}\}$. We have already seen that $(\nu_m(M_v),\nu_c(M_v))=(0,\ell)$ for each $v \in V'$. So, using \eqref{Equation:LambdaEvenReductionOneSmall2} and the fact that $b > (\ell+1)q a  + \ell(q-1)$, we can conclude that $\nu_b(M_v)=1$ for each $v \in V'$. It follows that $M_v=\{c^{[\ell]},b\} \uplus A_v$ for each $v \in V'$ where $\{A_v:v \in V'\}$ is a partition of $\{(\ell+1)qa_1,\ldots,(\ell+1)qa_{3q}\}$ into $q$ classes. So, because $\sigma(A_v) \equiv 0 \mod{(\ell+1)q}$ for each $v \in V'$ and $(\ell+1)q>\ell(q-1)$, it follows from \eqref{Equation:LambdaEvenReductionOneSmall2} that $\sigma(A_v) = (\ell+1)qa$ for each $v \in V'$. The existence of $\{A_v:v \in V'\}$ implies there is a partition of $\{a_1,\ldots,a_{3q}\}$ into $q$ classes such that the elements of each class sum to $a$.
\end{proof}

\begin{lemma}\label{Lemma:LambdaEvenNPComplete}
Let $\l$ be a positive even integer. For each $\alpha>1-\tfrac{2}{\l}(3-2\sqrt{2})$, \textsc{$(\lambda,\alpha)$-star decomp} is $\mathsf{NP}$-complete.
\end{lemma}

\begin{proof}
Let $\{a_1,\ldots,a_{3q}\}$ be an instance of \textsc{3-partition}. We will use Lemma~\ref{Lemma:LambdaEvenReduction} to reduce this instance to an instance of \textsc{$(\lambda,\alpha)$-star decomp}. Let $a=\frac{1}{q}(a_1+\cdots+a_{3q})$, $\alpha'=1-\tfrac{2}{\l}(3-2\sqrt{2})$, and $\ell=\frac{\l}{2}$.  It suffices to find integers $n$, $m$ and $r$ that satisfy the conditions of Lemma~\ref{Lemma:LambdaEvenReduction} and such that $m \leq \alpha(n-1)$ and $n$ is polynomial in $qa$.

We will first select integers $n$, $m$ and $r$ such that $n \gg (\ell+1)qa$, $n$ is polynomial in $qa$, $c$ (as defined as in Lemma~\ref{Lemma:LambdaEvenReduction}) is an integer, and
\begin{align*}
  r &= \tfrac{1}{\sqrt{2}}(n-1)+r'\\[1mm]
  m &= \alpha'(n-1) + m'
\end{align*}
for some $r'$ and $m'$ such that $r'= o(n)$, $m' = o(n)$, $m' \geq 2r' \geq 0$ and $m'>2\sqrt{2}(\ell+1)q (a+2)$. We proceed as follows.

\begin{description}
    \item[$\bm{n}$ and $\bm{x}$:]
Select $n \gg (\ell+1)qa$ so that $n$ is polynomial in $qa$ and $n-q$ has a divisor $x$ such that $x=o(n)$ and $\frac{n-q}{x}=o(n)$.
    \item[$\bm{p}$ and $\bm{r}$:]
Now let $p$ be the smallest prime such that $p \geq \tfrac{1}{x\sqrt{2}}(n-1)$ and note that $p \gg \ell$ because $x=o(n)$. Results on prime gaps imply that $p=\tfrac{1}{x\sqrt{2}}(n-1)+o(\frac{n}{x})$ (see \cite{BakHarPin}, for example). Choose $r=px$ and note this means $r'=o(n)$.
    \item[$\bm{m}$ (and hence $\bm{m'}$):]
We will now select an integer $m$ (and hence also a real $m'$) such that $m'>2r'$, $m'>2\sqrt{2}(\ell+1)q (a+2)$, $m'=o(n)$ and $z$ divides the integer $\frac{\ell(n-q)(n-q-1)}{x}-\frac{r(m-q)}{x}$ where $z$ is the integer $\frac{(\ell+1)(n-q)-2r}{x}$. Note this last fact will ensure that $(\ell+1)(n-q)-2r$ divides $\ell(n-q)(n-q-1)-r(m-q)$ and hence that $c$ (as defined as in Lemma~\ref{Lemma:LambdaEvenReduction}) is an integer. Noting $p=\frac{r}{x}$, we will be able to select such an $m$ provided that $\gcd(z,p)$ divides $\frac{\ell(n-q)(n-q-1)}{x}$ because then the diophantine equation $\beta z + \gamma p=\frac{\ell(n-q)(n-q-1)}{x}$ will have solutions for $\beta$ and $\gamma$ and we will be able to set $m=\gamma+q$ for an appropriately chosen solution $\gamma$. Observe that $\gcd(z,p) \in \{1,p\}$ since $p$ is prime. To see that $\gcd(z,p)$ divides $\frac{\ell(n-q)(n-q-1)}{x}$, note that $\gcd(z,p)$ must divide $\frac{(\ell+1)(n-q)}{x}$ since $\frac{(\ell+1)(n-q)}{x}=z+2p$. But then $\gcd(z,p)$ must in fact divide $\frac{n-q}{x}$, because $\gcd(\ell+1,p)=1$ since $p\gg\ell+1$ is prime.
    \item[$\bm{c}$ and $\bm{b}$:]
Choose $c$ and then $b$ according to their definitions in Lemma~\ref{Lemma:LambdaEvenReduction}.
\end{description}
So we can indeed select integers with the properties we claimed. Note that $m \leq \alpha(n-1)$ because $m=\alpha'(n-1)+o(n)$ and $\alpha>\alpha'$.

Now, $r=\tfrac{1}{\sqrt{2}}(n-1)+o(n)$, $m =\alpha'(n-1) + o(n)$ and it can be calculated that
\[c=\left(1-\mfrac{2-\sqrt{2}}{2\ell}\right)(n-1)+o(n).\]
From this, it is routine to check that $2c>m+q$, $\ell(c-q)>(\ell-1)(m-q)$ and $m>c+b+(\ell+1)qa$.

Finally, using the definitions of $m$ and $r$, we have $b>(\ell+1)qa+\ell(q-1)$ because
\begin{align*}
b-(\ell+1)qa-\ell(q-1)&=(\ell-1)(n-c-1)+r+q-1-c-2(\ell+1)qa-\ell(q-1)\\
	&= \mfrac{\ell\sqrt{2}m'+q(4-3\sqrt{2})}{2(\ell+1-\sqrt{2})}-\tfrac{2-\sqrt{2}}{2}-2(\ell+1)qa-\ell(q-1) \\
&\qquad+\mfrac{2r'((\ell+1)(\ell m'-2r')+2\sqrt{2}r')}{2(\ell+1-\sqrt{2})^2n}+o(1)
\end{align*}
and this can be seen to be positive using $m' \geq 2r' \geq 0$ and $m'>2\sqrt{2}(\ell+1)q (a+2)$.
\end{proof}

In this $\lambda$ even case, the greedy assignment method that we used for $\lambda$ odd will not suffice for establishing that every instance of \textsc{$(\lambda,\alpha)$-star decomp} is feasible when $\alpha \leq 1-\tfrac{2}{\l}(3-2\sqrt{2})$. Instead, we now introduce an alternative assignment method and, in Lemma~\ref{Lemma:GoodAllocationProperties}, establish some of its properties.

Let $m_1,\ldots,m_t$ be positive integers and $V$ be a set of $n$ vertices. An \emph{equitable assignment of $m_1,\ldots,m_t$ to multisets $\{M_v:v\in V\}$} is one for which
\begin{itemize}
    \item[(E1)]
for each $v \in V$, $|M_v| \in \{\lfloor\frac{t}{n}\rfloor,\lceil\frac{t}{n}\rceil\}$; and
    \item[(E2)]
for any $u,v \in V$ such that $\sigma(M_u) < \sigma(M_v)$ and any elements $x \in M_u$ and $y \in M_v$ such that $x < y$,  $\sigma(M_v) - \sigma(M_u) \leq y-x$.
\end{itemize}

\begin{lemma}\label{Lemma:GoodAllocationProperties}
Let $\l$ be a fixed even integer. Let $n$, $m$ and $m_1,\ldots,m_t$ be positive integers such that $n > m \geq m_1 \geq \cdots \geq m_t$, $m_{t-1}+m_t > m > \frac{n}{2}$, and $m_1+\cdots+m_t=\l\binom{n}{2}$. Let $V$ be an index set of cardinality $n$. An equitable assignment $\{M_v:v \in V\}$ of $m_1,\ldots,m_t$ can be found in polynomial time in $n$ and for any such assignment the following hold.
\begin{itemize}
    \item[\textup{(E3)}]
For any $u,v \in V$ and any element $x \in M_u$, $\sigma(M_v) \leq \max(\lceil\frac{t}{n}\rceil x,\sigma(M_u)+m-x)$.
    \item[\textup{(E4)}]
For any $u,v \in V$, $\sigma(M_v) \leq \sigma(M_u)+m$.
    \item[\textup{(E5)}]
For any $u,v \in V$ such that $|M_v| \geq 2$, then  $\sigma(M_v) \leq \frac{|M_v|}{|M_v|-1}\sigma(M_u)$.
\end{itemize}
\end{lemma}

\begin{proof}
We first show that an equitable assignment exists and can be found in polynomial time. Begin with any assignment $\{M_v:v \in V\}$ satisfying (E1). Clearly the variance $z=\sum_{v \in V}(\sigma(M_v)-\frac{\l}{2}(n-1))^2$ of the sums of the multisets is an integer which is polynomial in $n$. Because $m_{t-1}+m_t > \frac{n}{2}$ implies that $t$ is linear in $n$, it can be checked in quadratic time in $n$ whether the assignment satisfies (E2). If it does not, there are $u,v \in V$ such that $\sigma(M_u) < \sigma(M_v)$ and elements $x \in M_u$ and $y \in M_v$ such that $x < y$ and $\sigma(M_v) - \sigma(M_u) > y-x$, and we can replace the initial assignment with the assignment $\{M'_v:v \in V\}$ obtained from $\{M_v:v \in V\}$ by exchanging an element of $M_u$ equal to $x$ with an element of $M_v$ equal to $y$. The new assignment satisfies (E1) and the variance of the sums of the multisets of this new assignment is an integer strictly smaller than $z$ because
\[|\sigma(M'_v)-\sigma(M'_u)|=|\sigma(M_v)-\sigma(M_u)-2(y-x)|<|\sigma(M_v)-\sigma(M_u)|.\]
Thus by iterating this process we will obtain an equitable assignment in polynomial time in $n$.

We now show that (E3), (E4) and (E5) hold for any equitable assignment $\{M_v:v \in V\}$ of $m_1,\ldots,m_t$.
\begin{itemize}
    \item[\textup{(E3)}]
We may assume that $\sigma(M_u) < \sigma(M_v)$ for otherwise it is clear that $\sigma(M_v) \leq \sigma(M_u)+m-x$. If $\max(M_v) \leq x$, then
$\sigma(M_v) \leq |M_v|x \leq \lceil\frac{t}{n}\rceil x$.
Otherwise, $\max(M_v) \geq x$ and, by (E2),
\[\sigma(M_v) \leq \sigma(M_u)+\max(M_v)-x \leq \sigma(M_u)+m-x.\]
    \item[\textup{(E4)}]
Let $x=\min(M_u)$ and note that $x \leq \frac{1}{|M_u|}\sigma(M_u)$. By (E3), $\sigma(M_v) \leq \max(\lceil\frac{t}{n}\rceil x,\sigma(M_u)+m-x)$. We have $\lceil\frac{t}{n}\rceil x \leq (|M_u|+1)x \leq \sigma(M_u)+m$ and, clearly, $\sigma(M_u)+m-x<\sigma(M_u)+m$.
    \item[\textup{(E5)}]
We may assume that $\sigma(M_u) < \sigma(M_v)$ for otherwise (E5) follows immediately. Let $x=\min(M_u)$ and note that $x \leq \frac{1}{|M_u|}\sigma(M_u)$. If $x \geq \max(M_v)$, then
\[\sigma(M_v) \leq |M_v|x \leq \tfrac{|M_v|}{|M_u|}\sigma(M_u) \leq \tfrac{|M_v|}{|M_v|-1}\sigma(M_u).\]
If $\frac{m}{2} < x < \max(M_v)$, then $\max(M_v)-x < x$ and so by (E2),
\[\sigma(M_v) \leq \sigma(M_u) + \max(M_v)-x < \sigma(M_u) + x \leq \tfrac{|M_u|+1}{|M_u|}\sigma(M_u) \leq \tfrac{|M_v|}{|M_v|-1}\sigma(M_u).\]
If $x \leq \frac{m}{2}$, then $x < \min(M_v)$ because $x + \min(M_v) > m$. So by (E2)
\[\sigma(M_v) \leq \sigma(M_u) + \min(M_v) - x < \sigma(M_u) + \tfrac{1}{|M_v|}\sigma(M_v)\]
and, rearranging, we have $\sigma(M_v) < \frac{|M_v|}{|M_v|-1}\sigma(M_u)$. \qedhere
\end{itemize}
\end{proof}

As an example of when greedy assignment would be insufficient, consider attempting to decompose $4K_{100}$ into stars of sizes $\{90^{[166]},73^{[36]},72^{[31]}\}$. Greedy assignment would result in multisets $\{M_v:v \in V\}$ such that $M_u=\{90,73\}$ for some $u \in V$ and either $M_v=\{90^{[2]}\}$ or $|M_v| \geq 3$ for each $v \in V \setminus \{u\}$. The restriction function that takes the value $0$ at $u$ and the value $|M_v|$ at each $v \in V \setminus \{u\}$ shows that a decomposition with this assignment does not exist. However, because $90 \leq \alpha'(n-1)$, we will see in Lemma~\ref{Lemma:LambdaEvenExistence} below that a decomposition of $4K_{100}$ into stars of the prescribed sizes does indeed exist when the star sizes are assigned equitably.

We are now ready to show that, when $\alpha\leq 1-\tfrac{2}{\l}(3-2\sqrt{2})$, every instance of \textsc{$(\lambda,\alpha)$-star decomp} is feasible. We show this by first equitably assigning (as in Lemma~\ref{Lemma:GoodAllocationProperties}) the specified star sizes to the vertices of $\l K_n$. We then ``compress'' the resulting list at each vertex so as to reduce the number of possible restriction functions we need to consider. Finally we use Theorem~\ref{Theorem:RealTruthLambda} and Lemma~\ref{Lemma:MinimalFProperties1LambdaEven} to show the existence of a decomposition into the compressed sizes, and hence also of the desired decomposition. Again, we are forced to consider a number of cases.

\begin{lemma}\label{Lemma:LambdaEvenExistence}
Let $\l$ be a positive even integer and let $\alpha'=1 - \frac{2}{\lambda}(3-2\sqrt{2})$. For each positive integer $n$, if $M$ is a multiset of integers such that $\sigma(M)=\l\binom{n}{2}$ and \mbox{$\max(M) \leq  \alpha'(n-1)$}, then a decomposition of $\lambda K_n$ into stars of sizes given by the elements of $M$ exists and can be found in polynomial time in $n$.
\end{lemma}

\begin{proof}
The result is obvious if $n \leq 4$, so we may suppose that $n \geq 5$. Let  $m=\lfloor\alpha'(n-1)\rfloor$ and let $t=|M|$. By Lemma~\ref{Lemma:bloodyObvious} we may assume that $x+y>m$ for any distinct (but possibly equal) $x,y \in M$. Let $\ell=\frac{\l}{2}$ and let $\beta=3-2\sqrt{2} \approx 0.172$ so that $\alpha'=1-\frac{\beta}{\ell}$. Let $V$ be a set of $n$ vertices. By Lemma~\ref{Lemma:GoodAllocationProperties}, an equitable assignment $\{M_v:v \in V\}$ of the elements of $M$ to multisets can be produced in polynomial time.
Because $\sigma(M)=\ell n(n-1)>\ell mn$, there must be a vertex $u \in V$ such that $\sigma(M_u)>\ell m$ and hence, by (E4), $\sigma(M_v) > (\ell-1) m$ for each $v \in V$.

In each of two cases below we will define, for each $v \in V$, a ``compressed'' multiset $M^*_v$ of integers from $\{1,\ldots,m\}$ such that $\sigma(M^*_v)=\sigma(M_v)$ and $\sigma_i(M^*_v) \geq \sigma_i(M_v)$ for each $i \in \{1,\ldots,|M^*_v|\}$. As discussed in Section~\ref{Section:Strategy}, by Lemma~\ref{Lemma:multistarDecomp} it will suffice to find a star $\mathcal{K}$-decomposition where $\mathcal{K}$ is the multigraph $\l K_V$ equipped with the multisets $\{M^*_v:v \in V\}$. For each case we will define $f$ to be a minimal restriction function for $\mathcal{K}$ given by Lemma~\ref{Lemma:MinimalFProperties1LambdaEven} and, for each nonnegative integer $i$, let $V_i=f^{-1}(i)$ and $n_i=|V_i|$. By Theorem~\ref{Theorem:RealTruthLambda}, it will suffice to show $\Delta_f(\mathcal{K}) \geq 0$ (note that Theorem~\ref{Theorem:RealTruthLambda} guarantees a polynomial time construction).

Note that because $\sigma(M^*_v) = \sigma(M_v)$ for each $v \in V$, we will usually use $\sigma(M_v)$ in preference to $\sigma(M^*_v)$ for the sake of clean notation. The quantity $\ell(n-m-1)$ will be important throughout this proof and it will be useful to note that $\beta(n-1) \leq \ell(n-m-1) < \beta(n-1)+\ell$ because $m=\lfloor \alpha'(n-1) \rfloor$.

\textbf{Case 1.} Suppose that $\sigma(M_{v}) > (\ell+1)m$ for some $v\in V$. By (E4), we must have $km < \sigma(M_v) \leq (k+2)m$ for all $v\in V$ for some positive integer $k \geq \ell$. For each $v \in V$, let
\[M^*_v=\left\{
  \begin{array}{ll}
    \{m^{[k]},\sigma(M_v)-k m\} & \hbox{if $k m < \sigma(M_v) \leq (k+1) m$;} \\
    \{m^{[k+1]},\sigma(M_v)-(k+1) m\} & \hbox{if $(k+1) m < \sigma(M_v) \leq (k+2) m$.}
  \end{array}
\right.\]
Note $\{m^{[k]}\} \subseteq M^*_v$ for each $v \in V$. So, by Lemma~\ref{Lemma:MinimalFProperties1LambdaEven}(b), we can take $f$ to be a minimal restriction function such that $f(v)=0$ or $f(v) \geq \ell+1$ for each $v \in V$. Suppose for a contradiction that $\Delta_f(\mathcal{K})<0$. Note that $n_0 > 0$ for otherwise $\Delta^+_f(\mathcal{K})=\l\binom{n}{2}$, contradicting $\Delta_f(\mathcal{K})<0$. By Lemma~\ref{Lemma:dualThing},
\begin{equation}\label{Equation:LambdaEvenCase1}
\sum_{v \in V_0}\sigma(M_v) \leq \l\mbinom{n}{2}-\Delta^-_f(\mathcal{K})  < \l\mbinom{n_0}{2} + n_0(\ell-1)(n-n_0).
\end{equation}
It follows that $\sigma(M_w) < \ell(n-1)-(n-n_0)$ for some $w \in V_0$ and hence that $\sigma(M_v) \leq \ell(n-1)-(n-n_0)+m$ for each $v \in V$ by (E4). Thus $\Delta^-_f(\mathcal{K}) \leq (n-n_0)(\ell(n-1)-(n-n_0)+m)$. Adding this to the second and third expressions in \eqref{Equation:LambdaEvenCase1} we obtain the contradiction
\[\l\mbinom{n}{2}  < \l\mbinom{n}{2}-(n-n_0)(n-m) \leq \l\mbinom{n}{2}.\]

\textbf{Case 2.} Suppose that $\sigma(M_v) \leq (\ell+1)m$ for each $v \in V$. Recall that $\sigma(M_v) > (\ell-1) m$ for each $v \in V$. For each $v \in V$, let $y_v=\max(\sigma_{\ell}(M_v)-(\ell-1) m,\lceil\frac{1}{2}(\sigma(M_v)-(\ell-1) m)\rceil)$ and
\[M^*_v=\left\{
  \begin{array}{ll}
    \{m^{[\ell-1]},y_v\} & \hbox{if $\sigma(M_v)=(\ell -1)m +y_v$;} \\
    \{m^{[\ell-1]},y_v,\sigma(M_v)-(\ell-1) m -y_v\} & \hbox{if $\sigma(M_v) > (\ell -1)m +y_v$.}
  \end{array}
\right.\]
(Intuitively, $y_v$ is the smallest integer that ensures $\sigma_{\ell}(M^*_v) \geq \sigma_{\ell}(M_v)$ and $y_v \geq \sigma(M^*_v) - \sigma_{\ell}(M^*_v)$.)
Note that $|M^*_v| \in \{\ell,\ell+1\}$ for each $v \in V$. So, by Lemma~\ref{Lemma:MinimalFProperties1LambdaEven}, we can take $f$ to be a minimal restriction function satisfying (a)(i) and (a)(ii) of Lemma~\ref{Lemma:MinimalFProperties1LambdaEven}.
By (a)(i), $V_0 \cup V_{\ell} \cup V_{\ell+1}  = V$. By (a)(ii) and the definition of $M^*_v$, it can be seen that, for each $v \in V_{\ell}$, $\sigma_{\ell}(M^*_v)=\sigma_{\ell}(M_v)$ and we will use $\sigma_{\ell}(M_v)$ in preference $\sigma_{\ell}(M^*_v)$.

Suppose for a contradiction that $\Delta_f(\mathcal{K}) < 0$. Note that $n_0 > 0$ for otherwise $\Delta^+_f(\mathcal{K})=\l\binom{n}{2}$ and $n_0 < n$ for otherwise  $f$ is uniformly $0$, both of which contradict $\Delta_f(\mathcal{K})<0$. Let $w$ be an element of $V_0$ such that $\sigma(M_w) \leq \sigma(M_v)$ for all $v \in V_0$. By Lemma~\ref{Lemma:dualThing}
\begin{equation}\label{Equation:LambdaEvenCase2}
\sum_{v \in V_{0}} \sigma(M_v) \leq \l\mbinom{n}{2}-\Delta^-_f(\mathcal{K}) < \lambda\mbinom{n_{0}}{2}+\ell n_0n_{\ell}+(\ell-1)n_0n_{\ell+1}.
\end{equation}
It follows that
\begin{equation}\label{Equation:LambdaEvenCaseSigmaMax}
\sigma(M_w) < \ell(n-1)-n_{\ell+1}.
\end{equation}

Adding $\Delta^-_f(\mathcal{K})$ to the second and third expressions in \eqref{Equation:LambdaEvenCase2}, we have
\begin{equation}\label{Equation:LambdaEvenCase2sum}
\l\mbinom{n}{2} < \lambda\mbinom{n_{0}}{2}+\ell n_0n_{\ell}+(\ell-1)n_0n_{\ell+1}+\medop\sum_{v \in V_\ell}\sigma_{\ell}(M_v)+\medop\sum_{v \in V_{\ell+1}}\sigma(M_v).
\end{equation}
We now distinguish a number of cases. In each case we will bound $\sum_{v \in V_\ell}\sigma_{\ell}(M_v)+\sum_{v \in V_{\ell+1}}\sigma(M_v)$ and use \eqref{Equation:LambdaEvenCase2sum} to obtain a contradiction. The cases divide according to whether $n_{\ell+1} \leq \ell(n-m-1)$, whether $\lfloor\tfrac{t}{n}\rfloor \geq \ell+1$ and whether $\sigma(M_v) \leq \sigma(M_w) +\frac{m}{2}$ for each $v \in V_{\ell+1}$. Note that certainly $\lfloor\tfrac{t}{n}\rfloor \geq \ell$ because $\lambda\binom{n}{2}=\sigma(M)\leq t(n-1)$ implies that $\ell n \leq t$. It will help to remember that $|M_v| \in  \{\lfloor\tfrac{t}{n}\rfloor, \lceil\tfrac{t}{n}\rceil\}$ for each $v \in V$ by (E1).

\textbf{Case 2a.} Suppose that $n_{\ell+1} \leq \ell(n-m-1)$ or that $\lfloor\tfrac{t}{n}\rfloor \geq \ell +1$. We have $\sigma(M_v) \leq \frac{\ell+1}{\ell}\sigma(M_w)$ for each $v \in V_{\ell+1}$ by (E5). We next prove the claim that $\sigma_{\ell}(M_v) \leq \sigma(M_w)$ for each $v \in V_{\ell}$ (recall that $\sigma_{\ell}(M^*_v)=\sigma_{\ell}(M_v)$ for each $v \in V_{\ell}$).

If $n_{\ell+1} \leq \ell(n-m-1)$, then our claim is vacuously true because $V_{\ell}=\emptyset$ by Lemma~\ref{Lemma:MinimalFProperties1LambdaEven}(a)(i), noting that $\sigma_\ell(M^*_v) \leq \ell m$ for each $v \in V$. So we may assume that $\lfloor\tfrac{t}{n}\rfloor \geq \ell +1$. Let $s=\min(M_w)$ and let $v \in V_{\ell}$. If $\max(M_v) \leq s$, then $\sigma_{\ell}(M_v) \leq \ell s \leq \sigma(M_w)$, as required. If $\max(M_v) > s$, then $\sigma(M_v) \leq \sigma(M_w)+\max(M_v)-s$ by (E2) and thus $\sigma_{\ell}(M_v) \leq \sigma(M_w)+\max(M_v)-s-\min(M_v)$ because $|M_v| \geq \ell+1$ by (E1). So again $\sigma_{\ell}(M_v) \leq \sigma(M_w)$ because $\max(M_v) \leq m$ and $s+\min(M_v)>m$. Thus our claim holds.

Now, because $\sigma(M_v) \leq \frac{\ell+1}{\ell}\sigma(M_w)$ for each $v \in V_{\ell+1}$ and $\sigma_{\ell}(M_v) \leq \sigma(M_w)$ for each $v \in V_{\ell}$,  we have from \eqref{Equation:LambdaEvenCase2sum} that
\begin{align*}
\l\mbinom{n}{2} &< \lambda\mbinom{n_{0}}{2}+\ell n_0n_{\ell} + (\ell-1)n_0n_{\ell+1} + n_{\ell}\sigma(M_w) + \tfrac{\ell+1}{\ell}n_{\ell+1}\sigma(M_w) \\
&\leq \l\mbinom{n}{2}-n_{\ell+1}(\tfrac{1}{\ell}n_{\ell+1}+1) \leq \l\mbinom{n}{2},
\end{align*}
where the second inequality follows using $n_\ell=n-n_0-n_{\ell+1}$ and \eqref{Equation:LambdaEvenCaseSigmaMax}.

\textbf{Case 2b.} Suppose that $n_{\ell+1} > \ell(n-m-1)$, that $\lfloor\tfrac{t}{n}\rfloor = \ell$ and that $\sigma(M_v) \leq \sigma(M_w) +\frac{m}{2}$ for each $v \in V_{\ell+1}$. Using first $\sigma(M_v) \leq \sigma(M_w) +\frac{m}{2}$ for each $v \in V_{\ell+1}$ and the obvious fact that $\sigma_{\ell}(M_v) \leq \ell m$ for each $v \in V_{\ell}$, and then \eqref{Equation:LambdaEvenCaseSigmaMax}, we have from \eqref{Equation:LambdaEvenCase2sum} that
\begin{align*}
\l\mbinom{n}{2} &< \lambda\mbinom{n_{0}}{2}+\ell n_0n_{\ell} + (\ell-1)n_0n_{\ell+1} + \ell m n_{\ell} + n_{\ell+1}(\sigma(M_w) +\tfrac{m}{2}) \\
&\leq \l\mbinom{n}{2} - (n-n_{\ell+1})\ell(n-m-1) - n_0\left(n_{\ell+1}-\ell(n-m-1)\right) + \tfrac{1}{2}n_{\ell+1}(m-2n_{\ell+1}).
\end{align*}
Now, because $n_0 \geq 1$ and $n_{\ell+1} > \ell(n-m-1)$, we have $n_0(n_{\ell+1}-\ell(n-m-1)) \geq n_{\ell+1}-\ell(n-m-1)$, and hence
\begin{align*}
\l\mbinom{n}{2} &< \l\mbinom{n}{2} - (n-n_{\ell+1}-1)\ell(n-m-1) + \tfrac{1}{2}n_{\ell+1}\left(m-2n_{\ell+1}-2\right) \\
&\leq \l\mbinom{n}{2} -
 \beta(n-1)^2 + \tfrac{1}{2}n_{\ell+1}\left(\left(1+\tfrac{2\ell-1}{\ell}\beta\right)(n-1) - 2n_{\ell+1} - 2\right),
\end{align*}
where the second inequality is obtained using $\ell(n-m-1) \geq \beta(n-1)$ and $m \leq \alpha'(n-1) =(1-\frac{\beta}{\ell})(n-1)$.
Because this last expression is maximised when $n_{\ell+1}=\frac{1}{4}((1+\tfrac{2\ell-1}{\ell}\beta)(n-1) - 2)$ and $n \geq 5$, it is routine to obtain a contradiction.

\textbf{Case 2c.} Suppose that $n_{\ell+1} > \ell(n-m-1)$, that $\lfloor\tfrac{t}{n}\rfloor = \ell$ and that $\sigma(M_v) > \sigma(M_w) +\frac{m}{2}$ for some $v \in V_{\ell+1}$. Let $s$ be the $\ell$th greatest element in $M_w$. Note that $s \leq \frac{1}{\ell}\sigma(M_w)$. We now state and prove some useful facts.
\begin{itemize}
    \item
$s > \frac{m}{2}$. This holds if $|M_w| = \ell+1$ because then $s \geq \min(M_w)$ and $s+\min(M_w) > m$. So suppose otherwise that $|M_w| = \ell$ and $s \leq \frac{m}{2}$. Then, by (E3), $\sigma(M_v) \leq \max(s(\ell+1),\sigma(M_w)+m-s)$ for each $v \in V$. So, because $s(\ell+1) \leq \frac{m}{2}(\ell+1) \leq \ell m$ and $\sigma(M_w)-s \leq (\ell-1)m$, we have $\sigma(M_v) \leq \ell m$ for each $v \in V$. This implies the contradiction $\sigma(M) < \l\binom{n}{2}$.
    \item
$\sigma_{\ell}(M_v) \leq \min(\ell m,\sigma(M_w)+m-s)$ for each $v \in V_{\ell}$. Obviously $\sigma_{\ell}(M_v) \leq \ell m$ for each $v \in V_{\ell}$. So suppose for a contradiction that $\sigma_{\ell}(M_u) > \sigma(M_w)+m-s$ for some $u \in V_{\ell}$. If $\max(M_u) \leq s$, then $\sigma_{\ell}(M_u) \leq \sigma(M_w)$ using $s \leq \frac{1}{\ell}\sigma(M_w)$. If $\max(M_u) > s$, then $\sigma(M_u) \leq \sigma(M_w)+\max(M_u)-s$ by (E2). So in either case we have a contradiction.
    \item
$\sigma(M_v) \leq (\ell+1)s$ for each $v \in V_{\ell+1}$. Let $u$ be an element of $V_{\ell+1}$ with a maximum value of $\sigma(M_u)$ and suppose for a contradiction that $\sigma(M_u) > (\ell+1)s$. Then, by (E3), $\sigma(M_u) \leq \sigma(M_w)+m-s$. Thus, because $s > \frac{m}{2}$, $\sigma(M_u) < \sigma(M_w)+\frac{m}{2}$ which contradicts the assumption of Case 2c.
\end{itemize}

Because $\sigma_{\ell}(M_v) \leq \min(\ell m,\sigma(M_w)+m-s)$ for each $v \in V_{\ell}$ and $\sigma(M_v) \leq (\ell+1)s$ for each $v \in V_{\ell+1}$, we have from \eqref{Equation:LambdaEvenCase2sum} that
\begin{equation}\label{Equation:LambdaEvenCase2d}
\l\mbinom{n}{2}  < \lambda\mbinom{n_{0}}{2}+\ell n_0n_{\ell} + (\ell-1)n_0n_{\ell+1} + n_\ell\min(\ell m,\sigma(M_w)+m-s) + (\ell+1)s n_{\ell+1}.
\end{equation}

Consider the right hand side of \eqref{Equation:LambdaEvenCase2d} as a function of $s$ on the domain $0 < s \leq \frac{1}{\ell}\sigma(M_w)$. The derivative of this function is $(\ell+1)n_{\ell+1}>0$ when $s < \sigma(M_w)-(\ell-1)m$ and is $(\ell+2)n_{\ell+1}+n_0-n$ when $s > \sigma(M_w)-(\ell-1)m$, where we used the fact that $n_\ell=n-n_0-n_{\ell+1}$. So, if $(\ell+2)n_{\ell+1}+n_0 \geq n$, the function is nondecreasing and has a global maximum at $s = \frac{1}{\ell}\sigma(M_w)$ and if $(\ell+2)n_{\ell+1}+n_0 < n$, it has a global maximum at $s = \sigma(M_w)-(\ell-1)m$. We distinguish cases accordingly.

\textbf{Case 2c(i).} Suppose that $(\ell+2)n_{\ell+1}+n_0 \geq n$. By our discussion above we can take $s = \frac{1}{\ell}\sigma(M_w)$ in \eqref{Equation:LambdaEvenCase2d} and hence $\min(\ell m,\sigma(M_w)+m-s) \leq \frac{\ell-1}{\ell}\sigma(M_w)+m$. Using these facts and $n_\ell=n-n_0-n_{\ell+1}$,  we have
\begin{align*}
\l\mbinom{n}{2}  &< \l\mbinom{n_0}{2}  + (m + \ell n_0) (n - n_0 - n_{\ell+1}) + (\ell - 1) n_0 n_{\ell+1}+ \tfrac{1}{\ell}((\ell - 1) (n - n_0) + 2 n_{\ell+1})\sigma(M_w) \\
 &< \l\mbinom{n}{2} + \tfrac{1}{\ell}(n-n_0-n_{\ell+1})\left(n_{\ell+1}-\ell(n-m-1)\right) - \tfrac{1}{\ell}n_{\ell+1}(n_{\ell+1}+\ell)\\
&\leq \l\mbinom{n}{2}-\tfrac{\beta}{\ell}(n-1)^2+\tfrac{1}{\ell}n_{\ell+1}((1 + \beta)(n-1) - \ell - 2n_{\ell+1}),
\end{align*}
where the second inequality is obtained by applying \eqref{Equation:LambdaEvenCaseSigmaMax} and the third is obtained by substituting $n_0 \geq 1$ (recall that $n_{\ell+1} > \ell(n-m-1)$ by the conditions of this case) and $\ell(n-m-1) \geq \beta(n-1)$ and then simplifying. This last expression is maximised when $n_{\ell+1}=\frac{1}{4}\left((1 + \beta)(n-1) - \ell\right)$ and using this we obtain
\[\l\mbinom{n}{2} < \l\mbinom{n}{2}  - \left(1-\tfrac{\sqrt{2}}{2}\right)(n-1) +\tfrac{\ell}{8} <\l\mbinom{n}{2}.\]
To see that the last inequality holds, note that using the condition of Case 2c and the fact that $m \leq n-2$, we have $n \geq  n_{\ell+1} > \ell(n-m-1) \geq \ell$.

\textbf{Case 2c(ii).} Suppose $(\ell+2)n_{\ell+1}+n_0 < n$. Together with $n_0 \geq 1$ and the condition of this case that $n_{\ell+1} > \ell(n-m-1)$, this implies that $\ell(\ell+2)(n-m-1)<n-1$ and hence that $\ell \in \{1,2,3\}$ (recall that $\beta(n-1) \leq \ell(n-m-1)$). By our discussion above we can take $s = \sigma(M_w)-(\ell-1)m$ in \eqref{Equation:LambdaEvenCase2d}. Using first this fact and $n_\ell=n-n_0-n_{\ell+1}$, and then \eqref{Equation:LambdaEvenCaseSigmaMax},
\begin{align}
\l\mbinom{n}{2}&< m(\ell n + n_{\ell+1}) - ((\ell+1)n_{\ell+1} + n_0)(n_{\ell+1} - \ell(n-m-1)) \nonumber \\
&\leq m(\ell n + n_{\ell+1}) - ((\ell+1)n_{\ell+1} + 1)(n_{\ell+1} - \ell(n-m-1)) \label{Equation:EvenLambdaCase2cii}
\end{align}
where the second inequality follows by using $n_0 \geq 1$ (recall that $n_{\ell+1} > \ell(n-m-1)$ by the conditions of this case). This last expression is maximised when $n_{\ell+1} = \frac{\ell}{2}(n-m-1) + \frac{m-1}{2(\ell+1)}$ and, using this together with $m = \lfloor\alpha'(n-1)\rfloor$, it is routine to obtain a contradiction by considering the cases $\ell =1$, $\ell=2$ and $\ell=3$ individually. In particular, note that if $\ell=3$ and $n=5$, then given $n_{\ell+1}$ must be an integer the expression \eqref{Equation:EvenLambdaCase2cii} is maximised for $n_{\ell+1}=2$  and in this case the right hand side of \eqref{Equation:EvenLambdaCase2cii} is equal to $\lambda\binom{n}{2}$ which is the required contradiction.
\end{proof}

\begin{proof}[\textbf{\textup{Proof of Theorem~\ref{Theorem:NPComplete} when $\lambda$ is even}}] If $\alpha \leq 1-\tfrac{2}{\l}(3-2\sqrt{2})$, Lemma~\ref{Lemma:LambdaEvenExistence} shows that every instance of \textsc{$(\lambda,\alpha)$-star decomp} is feasible and that the required decompositions can be constructed in polynomial time. If $\alpha > 1-\tfrac{2}{\l}(3-2\sqrt{2})$, Lemma~\ref{Lemma:LambdaEvenNPComplete} shows \textsc{$(\lambda,\alpha)$-star decomp} is $\mathsf{NP}$-complete.
\end{proof}

\section{Conclusion}

As mentioned in the introduction, the problems of when a complete $\l$-fold multigraph can be decomposed into matchings, paths or cycles of specified sizes have all been completely solved with numerical necessary and sufficient conditions (see \cite{Bar,Bry,BryHorMaeSmi}). This indicates that these kinds of problems can be tractable for families of graphs with low maximum degree. Stars arguably form the simplest example of a family of high degree graphs, so our results here suggest that, for high degree families, this style of decomposition problem will prove difficult.

Rather than bounding the maximum star size, one could attempt to find other sufficient conditions for the existence of a decomposition of a complete $\l$-fold multigraph into stars of specified sizes. A condition in the style of Lonc's \cite{Lon} would be one possibility. Of course, the $\mathsf{NP}$-completeness of the general problem indicates that finding neat necessary and sufficient conditions is unlikely. Finally, it would be interesting to determine whether results in the style of those in this paper could be obtained for the problem of decomposing a complete $\l$-fold multigraph into cliques of specified orders.

\medskip
\noindent\textbf{Acknowledgments.} Thanks to Ramin Javadi for pointing out an error in a previous version of the proof of $\mathsf{NP}$-completeness. He also informs us that he and Afsaneh Khodadadpoor have independently obtained proofs of Theorem~\ref{Theorem:RealTruthLambda} and of the $\mathsf{NP}$-completeness of \textsc{$(\lambda,1)$-star decomp} that they now do not intend to publish. Thanks to Peter Dukes and Gary MacGillivray for useful discussions and to the referees for their thorough reading and useful comments. This work was supported by an AARMS Postdoctoral Fellowship and Australian Research Council grants DP150100506 and FT160100048.

\bibliographystyle{acm}

\end{document}